\documentclass[12pt,a4paper]{amsart}

\usepackage[margin=1.2in]{geometry}

\usepackage{amsmath} 
\usepackage{amsthm} 
\usepackage{amsfonts} 
\usepackage{amssymb}

\usepackage{tikz}

\newcommand\ganz{\mathbb Z} 
\newcommand\real{\mathbb R}

\newcommand\gen{\text{\it Span}}

\renewcommand\({\left(}
\renewcommand\){\right)}

\newcommand\ssc{\scriptstyle}

\newcommand\la{\langle}
\newcommand\ra{\rangle}

\newcommand\ov{\overline}

\newcommand\mc{\mathcal}
\newcommand\mf{\mathfrak}

\newtheorem*{thm*}{Theorem}
\newtheorem{thm}{Theorem}[section]
\newtheorem{pro}[thm]{Proposition}
\newtheorem{lem}[thm]{Lemma}
\newtheorem{cor}[thm]{Corollary}
\newtheorem*{cor*}{Corollary}

\theoremstyle{definition}\newtheorem{defi}{Definition}[section]
\theoremstyle{definition}\newtheorem{rem}{Remark}[section]
\theoremstyle{definition}
\theoremstyle{definition}\newtheorem{exa}{Example}[section]

\newcommand\Anfi
{\draw{(0,0)--(1.5,0)
           (2.5,0)--(3.5,0)
           (6.5,0)--(8,0)};
\draw[dashed]{ (1.5,0)--(2.5,0)(3.5,0)--(6.5,0)};
\foreach \x in {0, 1, 3, 7, 8}
\draw[fill=white]{(\x,0) circle(3pt)};}

\newcommand\An
{\draw{(0,0)--(4,1)--(8,0)};
\draw[fill=black]{(4,1) circle(3pt)};
\Anfi}

\newcommand\Bnf
{\draw{(1,0)--(2,0)--(2.5,0) 
 (5.5,0)--(6,0)--(7,0)};
\draw[dashed]{ (2.5,0)--(5.5,0)};
\draw{(7,-1.5pt)--(8,-1.5pt)};
\draw{(7,1.5pt)--(8,1.5pt)};
\draw{(7.4,.15)--(7.55,0)--(7.4,-.15)};
\foreach \x in {1,2,6,7,8}
\draw[fill=white]{(\x,0) circle(3pt)};}

\newcommand\Bn
{\draw{ (.3,-.7)--(1,0)--(.3,.7)};
\draw[fill=white]{(.3,.7) circle(3pt)};
\draw[fill=black]{(.3,-.7) circle(3pt)};
\Bnf}

\newcommand\Cnf
{\draw{(1,0)--(2,0)--(2.5,0) (5.5,0)--(6,0)--(7,0)};
\draw[dashed]{ (2.5,0)--(5.5,0)};
\draw{(7,-1.5pt)--(8,-1.5pt)};
\draw{(7,1.5pt)--(8,1.5pt)};
\draw{(7.6,.15)--(7.45,0)--(7.6,-.15)};
\foreach \x in {1,2,6,7,8}
\draw[fill=white]{(\x,0) circle(3pt)};}

\newcommand\Cn
{\draw{(0,-1.5pt)--(1,-1.5pt)};
\draw{(0,1.5pt)--(1,1.5pt)};
\draw{(.4,.15)--(.55,0)--(.4,-.15)};
\draw[fill=black]{(0,0) circle(3pt)};
\Cnf}

\newcommand\Dnf
{\draw{(1,0)--(2,0)--(2.5,0) 
(5.5,0)--(6,0)--(7,0)--(7.7,.7) (7,0)--(7.7,-.7)};
\draw[dashed]{ (2.5,0)--(5.5,0)};
\foreach \x in {1,2,6,7}
\draw[fill=white]{(\x,0) circle(3pt)};
\draw[fill=white]{(7.7,.7) circle(3pt)};
\draw[fill=white]{(7.7,-.7) circle(3pt)};}

\newcommand\Dn
{\draw{ (.3,-.7)--(1,0)--(.3,.7)};
\draw[fill=white]{(.3,.7) circle(3pt)};
\draw[fill=black]{(.3,-.7) circle(3pt)};
\Dnf}

\newcommand\Eseif
{\foreach \x in {0, 1, 2,3}
\draw{(\x,1)--(\x+1,1)};
\draw{(2,1)--(2,0)};
\foreach \x in {0, 1, 2,3,4}
\draw[fill=white]{(\x,1) circle(3pt)};
\draw[fill=white]{(2,0) circle(3pt)};}

\newcommand\Esei
{\draw{(2,0)--(2,-1)};
\draw[fill=black]{(2,-1) circle(3pt)};
\Eseif}

\newcommand\Esettef
{\foreach \x in {1, 2,3,4,5}
\draw{(\x,0)--(\x+1,0)};
\draw{(3,0)--(3,-1)};
\foreach \x in {1, 2,3,4,5,6}
\draw[fill=white]{(\x,0) circle(3pt)};
\draw[fill=white]{(3,-1) circle(3pt)};}

\newcommand\Esette
{\draw{(0,0)--(1,0)};
\draw[fill=black]{(0,0) circle(3pt)};
\Esettef}

\newcommand\Eottof
{\foreach \x in {0,1,2,3,4,5}
\draw{(\x,0)--(\x+1,0)};
\draw{(2,0)--(2,-1)};
\foreach \x in {0,1,2,3,4,5,6}
\draw[fill=white]{(\x,0) circle(3pt)};
\draw[fill=white]{(2,-1) circle(3pt)};}

\newcommand\Eotto
{\draw{(6,0)--(7,0)};
\draw[fill=black]{(7,0) circle(3pt)};
\Eottof}

\newcommand\Fquattrof
{\foreach \x in {1,3}
\draw{(\x,0)--(\x+1,0)};
\draw{(2,1.5pt)--(3,1.5pt)};
\draw{(2,-1.5pt)--(3,-1.5pt)};
\draw{(2.4,.15)--(2.55,0)--(2.4,-.15)};
\foreach \x in {1,2,3,4}
\draw[fill=white]{(\x,0) circle(3pt)};}
\newcommand\Fquattro
{\draw{(0,0)--(1,0)};
\draw[fill=black]{(0,0) circle(3pt)};
\Fquattrof}

\newcommand\Gduef
{\foreach\y in{0,2,-2}
\draw{(0,\y pt)--(1,\y pt)};
\draw{(.6,.15)--(.45,0)--(.6,-.15)};
\foreach \x in {0,1}
\draw[fill=white]{(\x,0) circle(3pt)};}

\newcommand\Gdue
{\draw{(1,0)--(2,0)};
\draw[fill=black]{(2,0) circle(3pt)};
\Gduef}

\title{Root polytopes and Borel subalgebras }

\author{Paola Cellini}
\address{Dipartimento di Ingegneria e Geologia, Universit\`a di Chieti --
Pescara, Viale Pindaro 42, 65127 Pescara, Italy}
\email{pcellini@unich.it}

\author{Mario Marietti}
\address{Dipartimento di Ingegneria Industriale e Scienze Matematiche, Universit\`a Politecnica delle Marche, Via Brecce Bianche, 60131 Ancona,  Italy}
\email{m.marietti@univpm.it}

\begin{document}
\setlength{\baselineskip}{1.2\baselineskip}

\begin{abstract}
Let $\Phi$ be a finite crystallographic irreducible root system and 
$\mathcal P_{\Phi}$ be the convex hull of the roots in $\Phi$. 
We give a uniform explicit description of the polytope $\mathcal P_{\Phi}$, analyze the algebraic-combinatorial structure of its faces, and provide connections with the Borel subalgebra of the associated
Lie algebra. We also give several enumerative results.
\end{abstract}

\maketitle

{\it Keywords:} Root system; Root polytope; Weyl group; Borel subalgebra; Abelian ideal

\section{Introduction}

Let $\Phi$  be a finite crystallographic root system in an $n$-dimensional Euclidean space $\mathcal E$, with scalar product $(\,\cdot\,,\cdot\, )$.  
We denote by $\mathcal P_{\Phi}$ the convex hull  of all the roots in $\Phi$, and we call it the {\em root polytope} of $\Phi$.  
The aim of this paper is to give a uniform explicit description of the root polytope $\mathcal P_{\Phi}$. 
\par
The root polytopes, or some strictly related objects, are studied in some recents papers. 
Some authors (see in particular  \cite{K1} and \cite{K2}), intend by {\it root polytope} the convex hull of the positive roots together with the origin, first introduced in \cite{GGP} for the root system of type $A_n$. 
We call this the {\it positive} root polytope and, if confusion may arise, we call $\mathcal P_\Phi$ the {\it complete} root polytope. 
In this paper we only consider the complete root polytope; our results have direct applications  to the study of the positive root polytopes of types $A_n$ and $C_n$ (see \cite{CM2}). 
In \cite{ABHPS}, some properties of the complete root polytopes are provided
for the classical types through a case by case analysis, using the usual coordinate
descriptions of the root systems. 
In our paper, we give case free statements and proofs for all types.  
\par
We provide both a simple global description of the root polytope, and a clear analysis of the combinatorial structure of its faces. 
Our results have also a direct interesting application in the study of partition functions. 
More precisely, for all $\gamma$ in the root lattice, let $|\gamma|$ be the minimum number of roots needed to express $\gamma$ as a sum of roots.
Chiriv\`i uses the results in Sections 3 through 5 to prove several properties of the map $\gamma \mapsto |\gamma|$; in particular, the map is piecewise quasi-linear with the cones over the facets of $\mathcal P_{\Phi}$ as quasi-linearity domains (see  \cite{Ch}). Finally, our results give information about $\bigcup\limits_{w\in W} w(A)$, where $A$ is the fundamental alcove of the affine Weyl group of $\Phi$, since this union set is the polar polytope of $\mathcal P_{\Phi}$  (see \cite{CM3}).
\par
 The Weyl group $W$ of $\Phi$ acts on $\mathcal P_\Phi$ thus, in order to describe $\mathcal P_\Phi$, it is natural to describe the orbits of its faces, of all dimensions, under the action of $W$. 
And, in order to describe the faces, we may describe a special representative in each $W$-orbit.
\par
The choice of a root basis provides a special set of faces of $\mathcal P_\Phi$. 
We study these faces and then prove that they are representatives of the $W$-orbits.
Let $\Pi=\{\alpha_1, \dots, \alpha_n\}$ be a root basis and let  $\breve\Omega=\{\breve\omega_1,\dots, \breve\omega_n\}$ be the dual basis of $\Pi$ in $\mathcal E$ i.e., the corresponding set of fundamental coweights.
 Moreover, let $\Phi^+$ be the set of positive roots with respect to $\Pi$. 
If $\theta$ is the highest root of $\Phi$ with respect to $\Pi$,  and $\theta=\sum\limits_{i=1}^n m_i\alpha_i$, then each  hyperplane $(x, \breve\omega_i)=m_i$, for $i=1, \dots, n$, supports a face $F_i$ of $\mathcal P_\Phi$ of some dimension, that contains $\theta$. 
Thus for each $I\subseteq \{1, \dots, n\}$, the intersection $F_I$ of the  faces $F_i$ with $i$ in  $I$ is a face of $\mathcal P_\Phi$.  
We also set $F_\emptyset=\mathcal P_\Phi$. 
We call the $F_I$, for all $I\subseteq \{1, \dots, n\}$, the {\it standard parabolic faces}.
Let $V_I$ be the set of roots in $F_I$: $V_I=F_I\cap \Phi$.
For all nonempty $I$, the sets $V_I$ are readily seen to be filters, or {\it dual order ideals}, in the poset $\Phi^+$  with the usual order: $\alpha\leq\beta$ if and only if $\beta-\alpha$ is a sum of positive roots. 
Each filter of $\Phi^+$ has a natural algebraic interpretation. 
Let  $\mathfrak g$ be a complex simple Lie algebra  having root system $\Phi$ with respect to a Cartan subalgebra $\mathfrak h$. 
For all $\alpha\in \Phi$, let $\mathfrak g_\alpha$ be the root space corresponding to $\alpha$,  and let $\mathfrak b$ be the standard Borel subalgebra of $\mathfrak g$ corresponding to $\Phi^+$: $\mathfrak b=\bigoplus_{\alpha\in \Phi^+}\mathfrak g_\alpha\oplus \mathfrak h$.  
For any filter V of $\Phi^+$, $\bigoplus_{\alpha\in V}\mathfrak g_\alpha$ is an ideal of $\mathfrak  b$. 
Conversely, each ideal of $\mathfrak b$ included in $\bigoplus_{\alpha\in \Phi^+}\mathfrak g_\alpha$ is obtained in such a way. 
We say that $V$ is an {\it abelian} dual order ideal  if the corresponding  $\mathfrak b$-ideal is abelian. 
Moreover, we say that $V$ is {\it principal} if the corresponding ideal is: this amounts to say that $V$ has a minimum.
A useful technique for dealing with the abelian or with the  {\it ad}-nilpotent ideals of a Borel subalgebra is to see them as special subsets of the affine root system associated to $\Phi$ (see \cite{C-P-Alg}). 
The idea, for the abelian ideals, is due to D. Peterson, and was first described in \cite{K}.
We analyze the sets $V_I$ in this way. We denote by $\widehat \Phi$ the affine root system associated to $\Phi$ and by $\widehat \Pi$ an extension of $\Pi$ to a simple system of $\widehat\Phi$: $\widehat\Pi=\Pi\cup \{\alpha_0\}$. 
For each $\Gamma \subseteq \Pi$, we denote by $\Phi(\Gamma)$ the standard parabolic subsystem of $\Phi$ generated by $\Gamma$, and by $W\la\Gamma\ra$ the standard parabolic subgroup of the Weyl group generated by the reflections with respect to the roots in $\Gamma$. 
Similarly, for $\Gamma \subseteq \widehat\Pi$, we denote by $\widehat \Phi(\Gamma)$ the standard parabolic subsystem of $\widehat\Phi$ generated by $\Gamma$. For each $I\subseteq \{1, \dots, n\}$, we set  $\Pi_I=\{\alpha_i\mid i\in I\}$.
Our first results are the following (see Lemma \ref{affine}, Proposition \ref{laminimaelunga}, and Corollary \ref{numerovertici}). 

\begin{thm*} 
Let $I \subseteq \{1, \dots, n\} $, $F_I=\{x\in \mathcal P_{\Phi}\mid
(x,\breve{\omega}_{i})= m_i, \; \forall i \in I \} $,  and  $V_I =F_I\cap \Phi$.
Then $F_{I}$ is a face  of $\mathcal P_{\Phi}$, and $V_I$ is a principal abelian dual order ideal of $\Phi^+$. Moreover, 
\begin{enumerate}
\item if $I \neq \emptyset$,  $V_I$ is in bijection with the positive roots in 
 $\widehat{\Phi}(\widehat\Pi \setminus \Pi_I) \setminus \Phi(\Pi \setminus \Pi_I)$, and through this bijection,  the vertices of $F_{I}$ correspond to the positive long roots in this set; 
\item the number of vertices of $F_I$ is 
$[W\la \Pi \setminus \Pi_I \ra : W \la (\Pi \setminus \Pi_I) \cap \theta^{\perp} \ra ]$.
\end{enumerate}
\end{thm*}

In particular, the face $F_I$ depends only on the irreducible component of $\alpha_0$ in the Dynkin diagram of $\widehat \Pi\setminus\Pi_I$. 
Hence it is clear that different sets of indices may yield the same face.   
For each $I\subseteq \{1, \dots, n\}$, the set of all $J\subseteq  \{1, \dots, n\}$ such that $F_J=F_I$ is an interval in the Boolean algebra of the subsets of $\{1, \dots, n\}$. 
In particular, it has a maximum  and a minimum, that we describe explicitely. 
The dimension and  the  stabilizer of  $F_I$ can be directly computed from the maximum and the minimum, respectively. Moreover, the minimum provides information  about the barycenter of $F_I$.
We sum up these results in the following statement (see Propositions \ref{intervallo}, \ref{chiusi}, \ref{stabilizzatore}, and  \ref{baryst}).

\begin{thm*}
Let $I\subseteq\{1, \dots, n\}$. There exist two subsets $\partial I$ and $\ov I$ of $\{1, \dots, n\}$ such that 
$\{J \subseteq\{1, \dots, n\} \mid F_J=F_I\} = \{J \subseteq \{1, \dots, n\} \mid \partial I \subseteq J \subseteq \ov{I} \}$. Moreover,  
\begin{enumerate}
\item the dimension of $F_I$ is $|\Pi| - |\ov I|$, 
\item the stabilizer of $F_I$ in $W$ is $W\la \Pi \setminus \Pi_{\partial I} \ra$,
\item the barycenter of $F_I$ lies in the cone generated by  $\{\breve\omega_i\mid i\in \partial I\}$, 
\item $\{\ov I \mid I \subseteq \{1, \dots, n\} \} = \{ I \subseteq \{1, \dots, n\} \mid
\widehat{\Phi}(\widehat \Pi \setminus \Pi_I) \text{ is irreducible} \}$.
\end{enumerate}
In particular, the set of the standard parabolic faces is in bijection with the irreducible standard parabolic subsystems of the affine root system that contain the affine root $\alpha_0$.
\end{thm*}

One can easily check that the sets $\partial I$ have at most three roots. 
Thus each standard parabolic face $F_I$, of any dimension, can always be obtained as an intersection of one, two, or three faces of type $F_i$, $i\in \{1, \dots, n\}$. 
We call these the {\it coordinate faces}. 
By item (3) of the above theorem, the barycenter of any  $F_I$ is a positive linear combination of  at most three fundamental coweights. 
\par
Associating to each $F_I$ its minimal root, we obtain an injective  map from  the set of standard parabolic faces into the set of positive roots. We characterize the image of this map (Proposition \ref{carat-min}). 
\par
The analysis of the standard parabolic and coordinate faces is made in Sections 3 and 4.    
\par
In Section 5 we deal with general faces and prove that each face belongs to the $W$-orbit of a standard parabolic face. 
Thus, each face of $\mathcal P_\Phi$ corresponds to an abelian ideal of a Borel subalgebra of $\mathfrak g$. 
\par
Part of the  results of this section hold in the context of weight  polytopes and follows by Vinberg's results of \cite{V} (see also \cite{K-R} for a generalization). 
A weight polytope $P(\lambda)$ is the convex hull of the $W$-orbit of the weight $\lambda$ of the Lie algebra $\mathfrak g$.  
Since it is easy to prove that $\mathcal P_\Phi$ is the convex hull of the long roots, we have that $\mathcal P_{\Phi}=P(\theta)$, hence Vinberg's results specialize to the root polytope. 
More precisely, the fact that the orbits of the faces are in bijection with the irreducible subsystems of the affine root system that contain the affine root can be directly deduced from Vinberg's results. 
Since we obtain this fact as an easy consequence of the results of the previous sections and of a further result that does not hold in the general context of weight polytopes (Theorem \ref{nonspezzato}), we take our proofs independent from Vinberg's results. 
We state the main results of Section 5 in the next theorem and corollary. The theorem sums up  Proposition \ref{nonconiugate}, Theorem \ref{tutteparaboliche2}, and Corollary \ref{fpoly}; the corollary  corresponds to Corollary \ref{mezzi-spazi}.
\begin{thm*}
The faces $F_{I}$ form a complete set of representatives of the $W$-orbits.
Moreover, the $f$-polynomial of $\mathcal P_\Phi$ is
$$\sum\limits_{I \in \mathcal I }\,  [W:W\la \Pi \setminus \Pi_{\partial I} \ra] \,
t^{n - |I|}, $$ 
where $ \mathcal I = \{ I \subseteq \{1, \dots, n\} \mid \widehat{\Phi}(\widehat \Pi \setminus \Pi_I) \text{ is irreducible} \}$.
\end{thm*}
In particular, the orbits of the facets correspond to the simple roots of $\Phi$  that do not disconnect  the extended Dynkin graph. 
Thus  we obtain the following  explicit representation of $\mathcal P_{\Phi}$ as an intersection of a minimal set of half-spaces. 

\begin{cor*} 
Let 
$\Pi_\mathcal I=\left\{\alpha\in \Pi\mid \widehat\Phi\(\widehat \Pi\setminus\{\alpha\} \) \text{ is irreducible}\right\}$ 
and let $\mathcal L(W^\alpha)$ be a set of representatives of the left cosets of $W$ modulo the subgroup  $W\la\Pi\setminus\{\alpha\}\ra$.   
Then 
$$\mathcal P_\Phi=\big\{ x\mid (x, w\breve\omega_\alpha)\leq m_\alpha \ 
\text{for all }\alpha\in \Pi_\mathcal I \text{ and } w\in \mathcal
L(W^\alpha )\big\}.$$
Moreover, the above one is the minimal set of linear inequalities that defines $\mathcal P_\Phi$ as an intersection of half-spaces. 
\end{cor*}
 
Finally, in Section 6, we find the minimal faces that contain the short roots, and prove that they form a single $W$-orbit. 
Moreover, we study the 1-skeleton of $\mathcal P_\Phi$ and find the special property that either all  edges are long roots, or all  edges are the double of short roots.

\section{Preliminaries}

In this section, we fix the notation and recall the basic results that we most 
frequently use in the paper.
For basic facts about root systems, Weyl groups, Lie algebras, and convex
polytopes, 
we refer the reader, respectively, to \cite{Bou}, \cite{BB} and
\cite{Hum}, \cite{H}, and~\cite{G2}. 
\par
 Given $n,m\in {\mathbb Z}$, with $n\le m$, we let
$[n,m]=\{n,n+1,\dots,m\}$ and, for $n\in {\mathbb N}\setminus \{0\}$, we let
$[n]=[1,n]$. For every set $I$, we denote its cardinality by $|I|$. 
We write $:=$ when the term at its left is defined by the expression at its
right. We denote by $\gen_\real X$ the real vector space generated by $X$.
\par
Let $\Phi$ be a finite irreducible (reduced) crystallographic root system in the
real 
vector space $\gen_\real \Phi$ endowed with the positive definite
bilinear form $(\,\cdot\,,\cdot\,)$.
We fix our further notation on the root system and its Weyl group in the following list:
\smallskip 

{\renewcommand{\arraystretch}{1.2}
$
\begin{array}{@{\hskip-1.3pt}l@{\qquad}l}
n &  \textrm{the rank of $\Phi$}, 
\\
\Phi_{\ell} & \textrm{the set of long roots ($\Phi_{\ell}=\Phi$ in simply laced 
cases)}, 
\\
\Gamma_{\ell}  & = \Gamma \cap \Phi_{\ell}, \textrm{ for all $\Gamma \subseteq \Phi$},
\\
\Phi_{s} & \textrm{the set of short roots}, 
\\
\Gamma_{s}  & = \Gamma \cap \Phi_{s}, \textrm{ for all $\Gamma \subseteq
\Phi$},
\\
\textrm{$\alpha$-$(\beta)$} & =(\beta+\ganz \alpha)\cap \Phi,  \textrm{ the
$\alpha$-string through $\beta$},
\\
\Pi= \{\alpha_1, \ldots, \alpha_n\} &  \textrm{the set of simple roots}, 
\\
\breve\Omega=\{\breve{\omega}_1, \dots, \breve{\omega}_{n}\}  &  \textrm{the set of fundamental coweights (the dual basis of $\Pi$)},
\\
\Phi^+  &  \textrm{the set of positive roots w.r.t. $\Pi$},
\\
\widehat{\Phi}  &  \textrm{the affine root system associated with $\Phi$}, 
\\
\alpha_0  &  \textrm{the affine simple root of $\widehat{\Phi}$},
\\ 
\widehat{\Pi} & = \{\alpha_0\}\cup\Pi, 
\\
 \widehat{\Phi}^+  &  \textrm{the set of positive roots of $\widehat{\Phi}$
w.r.t.  $\widehat{\Pi}$}, 
\\
\Phi(\Gamma),\ \widehat\Phi(\Gamma) &  \textrm{the root subsystem generated by $\Gamma$ in $\Phi$, in $\widehat\Phi$, resp.}
\\
&\text{(for $\Gamma\subseteq \Phi$,  $\Gamma\subseteq \widehat\Phi$, resp.)},
\\
\Phi^+(\Gamma),\ \widehat\Phi^+(\Gamma)  &  = \Phi(\Gamma) \cap \Phi^+, \ \widehat\Phi(\Gamma) \cap \widehat\Phi^+, \text{resp.}
\\
c_i(\alpha)  &  \textrm{the $i$-th coordinate of $\alpha$ w.r.t. $\widehat \Pi$: $\alpha=\sum_{i=0}^nc_i(\alpha) \alpha_i$},
\\
\textrm{supp}(\alpha) & =\{\alpha_i \in \widehat \Pi \mid c_i (\alpha) \neq 0  \}, \text{ the support of $\alpha$},
\\
ht(\alpha)  &= \sum_{i=0}^{n}c_i(\alpha),  \textrm{ the height of the root
$\alpha$}, 
\\
\theta  & \textrm{the highest root in $\Phi$},
\\
\theta_s  & \textrm{the highest short root in $\Phi$},
\\
m_i &   =c_i(\theta) 
\\
W   &  \textrm{the Weyl group of $\Phi$},
\\
s_{\alpha}   & \textrm{the reflection with respect to $\alpha$},
\\
\ell  &  \textrm{the length function of $W$ w.r.t. $\Pi$},
\\
D_r(w) & =\{i\in [n] \mid \ell(ws_{\alpha_i})<\ell(w)\},  \textrm{ the right descent
set of $w$},
\\
w_0 & \textrm{the longest element of $W$ w.r.t. $\Pi$},
\\
W\la \Gamma\ra & \textrm{the subgroup of $W$ generated by 
$\{s_\alpha\mid \alpha\in \Gamma\}$ ($\Gamma\subseteq \Phi$)},
\\
\widehat{W} &  \textrm{the Weyl group of $\widehat{\Phi}$}.
\end{array}$
\bigskip

By \emph{the root poset of $\Phi$} (w.r.t. the basis $\Pi$) we intend  the
partially ordered set whose underlying set is $\Phi^+$, with the standard
order,} $\alpha \leq \beta$ if and only if $\beta - \alpha$ is a nonnegative
linear combination of roots in $\Phi^+$. 
The root poset could be equivalently defined as the transitive closure of the relation  $\alpha \lhd \beta$ if and only if 
$\beta - \alpha$ is a simple root.
The root poset hence is ranked by the height function and has the highest root
$\theta$ as maximum.
A dual order ideal is, as usual, a subset I of $\Phi^+$ such that, if $\alpha \in I$ and $\beta \geq \alpha$, 
then $\beta \in I$.
\par

For the reader convenience, we collect in the following propositions the
standard results on root systems that are 
frequently used in the paper, often without explicit mention.
\bigskip

\begin{pro}
\label{standard}
 Let $\Phi$ be any root system. If $\alpha, \beta \in \Phi$, $\alpha \neq -
\beta$, $(\alpha,\beta)<0$, then $\alpha + \beta \in \Phi$. Moreover,
\begin{enumerate}
 \item if $L$ is a subset of $\Pi$ which is connected in the Dynkin
diagram of
$\Phi$, then $\sum_{\alpha \in L} \alpha $ is a positive root,
\item the support of a root is connected in the Dynkin
diagram of
$\Phi$.
\end{enumerate}
\end{pro}

\smallskip

\begin{pro}[\cite{Bou}, Ch. VI, \S 1]
\label{standard2}
  Let $\Phi$ be any root system, and let $\alpha$ and $\beta$ be
non-proportional roots of $\Phi$.
Then the set $\{j\in \mathbb Z \mid \beta + j \alpha \in \Phi \}$ is an interval
$[-q,p]$ containing 0.
The $\alpha$-string through $\beta$, 
$\alpha$-$(\beta)$, has exactly 
$- \frac{2(\gamma, \alpha)}{(\alpha, \alpha)} + 1 $ roots, where $\gamma= \beta
- q \alpha$ is the origin of the string.
\end{pro}

\smallskip

\begin{pro}[\cite{Bou}, Ch. VI, \S 1, Proposition 2.4]
\label{standard3}
 Let $\Phi$ be any root system and let $\Phi '$ be the intersection of $\Phi$
with a subspace of $\gen_\real\Phi$. Then 
\begin{enumerate}
 \item $\Phi '$ is a root system in the subspace it spans;
\item given any basis $\Pi ' $ of $\Phi '$, there exists a basis of $\Phi$
containing $\Pi'$.
\end{enumerate}
\end{pro}

\bigskip

The following result might be less known than the previous ones and we give a proof of it.

\begin{pro}
\label{standard4}
 Let $\Phi$ be an irreducible root system such that the $i$-th coordinate 
$m_i$ of $\theta$ w.r.t. $\Pi$ is $1$. Then $\{-\theta\} \cup \Pi
\setminus \{\alpha_i\}$ is a basis of $\Phi$. Moreover, if $w_i$ is the longest
element in $W\la\Pi \setminus \{\alpha_i\}\ra$, then $w_i$ is an involution,
and 
$$w_i(\Pi
\setminus \{\alpha_i\})=-(\Pi \setminus \{\alpha_i\}), \qquad
w_i(\alpha_i)=\theta.$$
\end{pro}

\begin{proof}
The first statement follows from the next ones, since then  $\{-\theta\} \cup
\Pi
\setminus \{\alpha_i\}$ would be the opposite of $w_i(\Pi)$, which is a basis.
The fact that $w_i$ is an involution that maps $\Pi \setminus \{\alpha_i\}$ into
$-(\Pi \setminus \{\alpha_i\})$ is well known. Moreover $w_i$ permutes the
positive roots on $\Phi\setminus\Phi(\Pi\setminus\{\alpha_i\})$. Hence, 
$ht(w_i(\theta)) = \sum m_j
ht(w_i(\alpha_j)) = ht(w_i(\alpha_i))-\sum_{j\neq i} m_j
|ht(w_i(\alpha_j))|=ht(w_i(\alpha_i))-(ht(\theta)-1)$,
and $ht(w_i(\theta))>0$. 
This implies that $ht(w_i(\alpha_i))=ht(\theta)$ and $ht(w_i(\theta))=1$,
whence 
$w_i(\alpha_i)=\theta$, and  $w_i(\theta)=\alpha_i$.
\end{proof}
\smallskip

For the general theory of affine root systems, we refer the reader to \cite{Kac}. 
We briefly describe the (untwisted) affine root system $\widehat \Phi$ associated with $\Phi$. 
\par 
We extend $\gen_\real\Phi$ to  a $n+1$ dimensional real vector 
space $\gen_\real\Phi\oplus \real \delta$ and set 
$$\widehat\Phi=\Phi+\ganz \delta:=\{\alpha+k\delta\mid  \alpha\in \Phi,\ k\in \ganz\}.$$
Then $\widehat\Phi$ is an affine root system in $\gen_\real\Phi\oplus \real \delta$ endowed with the positive semidefinite symmetric bilinear form that extends the scalar product of $\gen_\real\Phi$ and has $\real \delta$ as its kernel.   
\par
If we take $\alpha_0=-\theta+\delta$, then $\widehat \Pi:=\{\alpha_0\}\cup \Pi$
is a root basis for $\widehat \Phi$. The set of positive roots of $\widehat
\Phi$  with respect to 
$\widehat \Pi$ is $\widehat \Phi^+:=\Phi^+\cup \(\Phi+\ganz^+\delta\)$, where
$\ganz^+$ is the set of positive integers.
\par
Let $\mathfrak{g}$ be a complex simple Lie algebra, and $\mathfrak h$  a Cartan 
subalgebra of $\mathfrak g$ such that $\Phi$ is the root system of $\mathfrak g$
with respect to $\mathfrak h$. For each $\alpha\in\Phi$, let
$\mathfrak{g}_{\alpha}$ be the root space of $\alpha$.
For every choice of a basis $\Pi$ of $\Phi$, we have the corresponding standard Borel subalgebra 
$\mathfrak b(\Pi):=\mathfrak h\oplus \sum_{\alpha \in \Phi^+} \mathfrak{g}_{\alpha}$. 
We let $\mathfrak b := \mathfrak b(\Pi)$ if no confusion arises.
Being $\mathfrak h$-stable, any ideal $\mathfrak{i}$ of $\mathfrak b$ is
compatible with the root space decomposition. 
Since, given $\alpha, \alpha' \in \Phi^+$, 
$[\mathfrak{g}_{\alpha}, \mathfrak{g}_{\alpha'}]$ is equal to 
$\mathfrak{g}_{\alpha + \alpha'}$ if $\alpha + \alpha' \in \Phi$ 
and is trivial otherwise, 
if  $\mathfrak{i} = \sum_{\alpha \in \Gamma} \mathfrak{g}_{\alpha}$ is an ideal of $\mathfrak b$, then 
 $\Gamma \subseteq \Phi^+$ satisfies $(\Gamma + \Phi^+) \cap \Phi \subseteq \Gamma$, or, equivalently, $\Gamma$ is a 
dual order ideal in the root poset. If
 we further require that $\mathfrak{i}$ be abelian, than  
 $\Gamma $ must satisfy also the abelian condition: $(\Gamma + \Gamma) \cap \Phi = \emptyset$. 
Indeed, all abelian ideals of $\mathfrak{b}$ are of this kind since they must be $ad$-nilpotent (i.e., included in 
$\sum_{\alpha \in \Phi^+} \mathfrak{g}_{\alpha}$).
By a {\em principal abelian ideal of   $\mathfrak{b}$}, we mean an abelian ideal  $\mathfrak{i}$ of the form
 $\mathfrak{i}=\sum_{\alpha \in \Gamma} \mathfrak{g}_{\alpha}$, where $\Gamma$, as a 
subposet of the root poset, has a minimum $\eta$ 
(hence $\Gamma$ is an interval since the highest root $\theta$ is the maximum). 
A principal abelian ideal is generated by 
any non-zero vector of the root space $\mathfrak{g}_{\eta}$.

\section{Standard Parabolic Faces}

In this section, we consider a set of distinguished faces of the
root polytope $P_\Phi$, and analyze their rich combinatorial structure.
\par 
A proper face
of $\mc P_\Phi$ is, by definition,  the intersection of  $\mc
P_\Phi$ with some affine hyperplane that does not split   $\mc P_\Phi$.
Moreover, any intersection of faces is a face. 
Recall that, in our notation, the highest root $\theta$ has $m_i$ as $i$-th
coordinate w.r.t. $\Pi$, i.e.,  $(\breve{\omega}_i, \theta)= m_i$. Hence, each
hyperplane $(\breve{\omega}_i,-)= m_i$ contains a face, of some dimension, of 
$\mc P_\Phi$. 
\smallskip

\begin{defi}
For each $I\subseteq[n]$ and $i\in [n]$,   we set  
$$F_I:=\{x\in  \mc P_\Phi\mid (\breve\omega_i, x)=m_i\ \forall i\in I\}\qquad 
\text{and } \qquad  F_i:=F_{\{i\}}.$$ 
Thus, the $F_i$ are proper faces of  $\mc P_\Phi$, all containing $\theta$. We
call them the {\it coordinate faces}.  
Moreover, $F_\emptyset=\mc P_\Phi$,  and, for all $I\neq\emptyset$, 
$F_I$ is a proper nonempty face,
since
$F_I=\cap_{i\in I} F_i$.  We call the face $F_I$, for $I\neq \emptyset$,
the {\it standard parabolic faces} of $P_\Phi$. 
We also set 
$$V_i:=F_i\cap \Phi,\qquad \text{and}\qquad V_I:=F_I\cap \Phi.$$  
It is clear that $F_I$ is the convex hull of $V_I$. 
\end{defi}
\smallskip

For any $I\subseteq[n]$, we set 
$$
\Pi_I:=\{\alpha_i \in \Pi \mid i\in I\}.
$$
\smallskip

Moreover, we denote by $\widehat\Phi_0(\widehat \Pi\setminus \Pi_I)$ the irreducible component of $\alpha_0$ in \hbox{$\widehat\Phi(\widehat \Pi\setminus \Pi_I)$}, and set 
$$ 
\widehat\Phi_0^+(\widehat \Pi\setminus \Pi_I):=\widehat\Phi_0(\widehat \Pi\setminus \Pi_I)\cap\widehat\Phi^+, \qquad (\widehat \Pi\setminus\Pi_I)_0:=\widehat\Phi_0(\widehat \Pi\setminus\Pi_I)\cap \widehat \Pi.$$
\bigskip

The following result provides a first connection of the root polytope with the
extended root system. 

\begin{lem}\label{affine}
Let $I\subseteq [n]$, $I\neq \emptyset$, and consider the subset 
$$-V_I+\delta:=\{-\alpha+\delta\mid \alpha\in V_I\}$$
of the affine root system $\widehat\Phi$.
Then 
$$-V_I+\delta=\widehat\Phi^+(\widehat \Pi\setminus
\Pi_I)\setminus\Phi=\widehat\Phi_0^+
(\widehat \Pi\setminus \Pi_I)\setminus\Phi.$$
\end{lem}

\begin{proof}
It is clear that the coordinate $c_0$ is constantly $1$ on $-\Phi^+ +\delta$. 
Since $\beta\in V_I$ if and only if $(\beta, \breve\omega_i)=(\theta,
\breve\omega_i)$ for all $i \in I$, and $\alpha_0=-\theta+\delta$, we obtain
that 
$-V_I+\delta$ is contained in the standard parabolic subsystem
$\widehat\Phi(\widehat \Pi\setminus \Pi_I)$ of
  $\widehat\Phi$. Indeed, since each root in  $-V_I+\delta$ has $\alpha_0$ in
its support, $-V_I+\delta$ is contained in the set of the positive
roots of the irreducible component of $\widehat\Phi(\widehat
\Pi\setminus \Pi_I)$ that contains $\alpha_0$. \par
Conversely, it is clear that each positive root in $\Phi_0(\widehat \Pi\setminus \Pi_I)$ that has $\alpha_0$ in its support belongs to $-V_I+\delta$.
\end{proof}
\medskip

\begin{pro}
\label{principale}
\label{laminimaelunga}
Let $I\subseteq [n]$, $I\neq \emptyset$. Then:
\begin{enumerate} 
\item
as a subposet of $\Phi^+$, $V_I$ has maximum and minimum, and both of them are
long roots;
\item
$V_I$ is an abelian dual order ideal, hence the subspace $\mathfrak{i}_{V_I} :=
\sum_{\alpha \in V_I} \mathfrak{g}_{\alpha}$ is a principal  abelian ideal of
$\mathfrak{b}$.
\end{enumerate}
\end{pro}

\begin{proof}
(1) It is clear that $\theta$ is the maximum of $V_I$. 
\par
Let  $\widehat\eta_I$ be the highest root of the irreducible component 
$\widehat\Phi_0(\widehat \Pi\setminus \Pi_I)$, 
and set 
$\eta_I:=- \widehat\eta_I +\delta$. By Lemma \ref{affine}, we directly obtain
that 
$\eta_I$ is the minimum of $V_I$.
\par 
(2) The fact that $V_I$ is an abelian dual order ideal follows by noting that
the functional $(\breve{\omega}_i,-)$ cannot take values $> m_i$ on the roots.
In fact, given $\alpha \in V_{I}$ and $\beta \in \Phi^+$, $\beta \geq \alpha$,
then $m_i=(\breve{\omega}_i,\alpha)\leq (\breve{\omega}_i,\beta)\leq m_i$, for
all $i\in I$: hence $\beta \in V_I$ and $V_I$ is a dual order ideal. The
abelianity follows by the fact that, for $\alpha, \alpha' \in V_I$, $\alpha +
\alpha'$ cannot be a root since $(\breve{\omega}_i, \alpha + \alpha') = 2m_i$.  
Since $V_I$ has a minimum, the abelian ideal of $\mf b$ corresponding to $V_I$
is principal.
\end{proof}
\medskip

\begin{rem}
For any subset $\Sigma$ of $\Pi$, the root subsystem 
$\widehat \Phi(\{\alpha_0\}\cup \Sigma)$, through the natural 
projection of $\gen_\real\widehat\Phi$ onto $\gen_\real\Phi$, maps onto the root
subsystem $\Phi(\{\theta\}\cup \Sigma)$ of $\Phi$. If $\Sigma$ is a proper
subset of $\Pi$, this is a bijection and a root system isomorphism. In
particular, $\{-\theta\}\cup \Sigma$ is a root basis for $\Phi(\{\theta\}\cup
\Sigma)$. 
It is clear that, with respect to this basis, the coordinate relative to
$-\theta$ is at most $1$,  for all roots in $\Phi(\{\theta\}\cup \Sigma)$.
Therefore, if $-\eta$ is the highest root, by Proposition \ref{standard4},
$\{\eta\}\cup \Sigma$ is a root basis, too.  The set of positive roots with respect
to this latter basis is $\Phi(\{\theta\}\cup \Sigma)\cap \Phi^+$, i.e.,
$\Phi^+(\{\theta\}\cup \Sigma)$, according to our notation.
\end{rem}

It is clear, from Lemma \ref{affine}, that the map $I\mapsto F_I$, that associates to the subset $I$ of $[n]$ the
corresponding standard parabolic face, is not injective, in general. 
In fact, this is an injective map only when $\Phi$ is of type $A_1$ or $A_2$. 
We determine explicitly, for each parabolic face $F$, the set of all $I$ such that $F=F_I$.
\par
For any $I\subseteq [n]$, we set 
$$\ov I:=\{k\mid \alpha_k \not\in (\widehat \Pi\setminus \Pi_I)_0\}$$
and
$$\partial I:=\{ j\mid \alpha_j\in \Pi_I, \text{ and }\exists\, \beta \in  (\widehat \Pi\setminus \Pi_I)_0 \text{ s. t. } \beta\not\perp\alpha_j\}.$$ 
We call $\ov I$ the {\it closure} and $\partial I$ the {\it border} of $I$. 
\par
By definition 
$$
(\widehat\Pi\setminus \Pi_I)_0=\widehat \Pi \setminus \Pi_{\ov I}.
$$
The closure $\ov I$ and the border $\partial I$ depend only on
$(\widehat\Pi\setminus \Pi_I)_0$, 
hence $\ov I = \ov{\partial I}$ and $\partial I=\partial \ov I$.  
\par

If we denote by $\Gamma$ the extended Dynkin diagram of $\Phi$, and by $\Gamma(\Sigma)$, the subdiagram of $\Gamma$ induced by $\Sigma$,  for any $\Sigma\subseteq\widehat \Pi$, then $\Pi_{\ov I}$ is the set of all simple roots exterior to the connected subdiagram $\Gamma((\widehat\Pi\setminus \Pi_I)_0)$, while $\partial I$ is the 
set of simple roots that are exterior and adjacent to $\Gamma((\widehat\Pi\setminus \Pi_I)_0)$. In this sense, $\partial I$ is indeed the border of $\ov I$. 

\begin{rem}
The map $I\mapsto \ov I$ has actually the properties of a topological closure operator on the power set of $[n]$. Indeed, it is clear from the definition that  $I\subseteq \ov I$, and  that $I=\ov I$ if and only if $\Gamma(\widehat\Pi\setminus \Pi_I)$ is connected. Hence we get  $\ov{\ov I}=\ov I$ and $\ov \emptyset =\emptyset$; moreover, $\ov {I\cup J}=\ov I\cup \ov J$,  for all $I, J\subseteq [n]$, since the Dynkin diagram of any finite system is a tree.
\end{rem}

We illustrate the definition of $\ov I$ and $\partial I$  in the following example. 
\bigskip

\begin{exa} Let $I=\{5, 7\}$ in $\Phi$ of type $B_9$. Then we see that 
$\partial I=\{5\}$ and $\ov I=\{5, \dots, 9\}$:
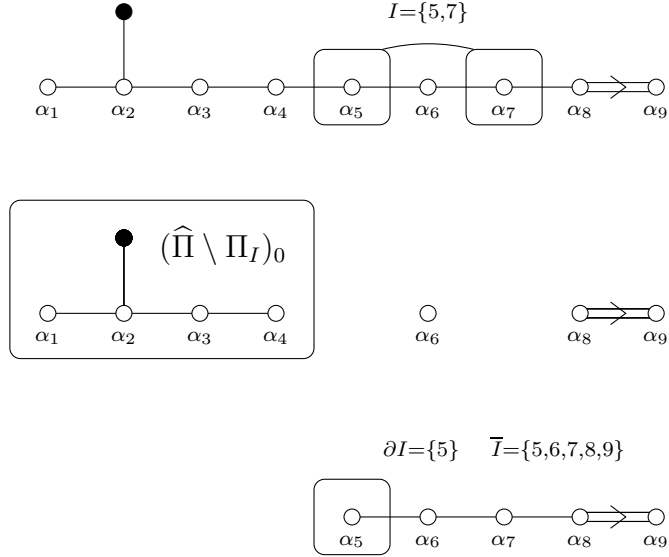
$$
\begin{tikzpicture}
\foreach \x in {2,3,4,5,6,7,8}
\draw{(\x-1,0)--(\x,0)}; 
\draw(2,0)--(2,1);
\draw(8,.06)--(9,.06);
\draw(8,-.06)--(9,-.06);
\draw(8.4,.15)--(8.6,0)--(8.4,-.15);
\foreach \x in {1,2,3,4,5,6,7,8,9}
{\draw[fill=white]{(\x,0) circle(3pt)};
\node[below]at(\x,-.1){$\ssc \alpha_\x$};}
\draw[fill=black]{(2,1) circle(3pt)};
\draw[rounded corners](4.5, .5)--(5.5, .5)--(5.5, -.5)--(4.5,-.5)-- cycle;
\draw[rounded corners](6.5, .5)--(7.5, .5)--(7.5, -.5)--(6.5,-.5)-- cycle;
\draw(5.4, .5)..controls(5.8, .6)and(6.2, .6)..(6.6, .5);
\node[above] at (6,.7){$\ssc I=\{5,7\}$};

\foreach \x in {2,3,4}
\foreach \y in {-3}
{
\draw(\x-1,\y)--(\x,\y); 
\draw(2,\y)--(2,\y+1);
\draw(8,\y+.06)--(9,\y+.06);
\draw(8,\y-.06)--(9,\y-.06);
}
\draw(8.4,-3-.15)--(8.6,-3)--(8.4,-3+.15);
\foreach \x in {1,2,3,4,6,8,9}
\foreach\y in{-3}
{\draw[fill=white](\x,\y) circle(3pt);
\node[below]at(\x,\y-.1){$\ssc \alpha_\x$};
\draw[fill=black](2,\y+1) circle(3pt);
}
\draw[rounded corners](.5,-3-.6)--(4.5,-3-.6)--(4.5, -3+1.5)--(.5,-3+1.5)-- cycle;
\node[above]at(3.3,-3+.5) {$(\widehat \Pi\setminus \Pi_I)_0$};
\foreach \y in {-5.7}
\foreach \x in {6,7,8}
{
\draw(\x-1,\y)--(\x,\y); 
\draw(8,\y+.06)--(9,\y+.06);
\draw(8,\y-.06)--(9,\y-.06);
}
\draw(8.4,-5.7-.15)--(8.6,-5.7)--(8.4,-5.7+.15);
\foreach \x in {5,6,7,8,9}
\foreach\y in{-5.7}
{\draw[fill=white](\x,\y) circle(3pt);
\node[below]at(\x,\y-.1){$\ssc \alpha_\x$};
}
\draw[rounded corners] (4.5,-5.7-.5)--(5.5,-5.7-.5)--(5.5, -5.7+.5)--(4.5,-5.7+.5)-- cycle;
\node[above]at(7,-5.7+.6){$\ssc\partial I=\{5\} \quad\ov I=\{5, 6, 7, 8, 9\}$};
\end{tikzpicture}
$$
\end{exa}
\bigskip

The following proposition follows directly from Lemma \ref{affine} and from the above definitions.  

\begin{pro}
\label{minmaxI}
\label{chiusuradigalois}
\label{biezione-st-di}
\label{intervallo}
Fix $I\subseteq [n]$ and let $J \subseteq [n]$. Then $F_J
= F_I$ if and only if  
$$\partial I \subseteq J \subseteq \ov I .$$
In particular,  the standard parabolic faces of $\mathcal P_{\Phi}$ are in  bijection with the connected subdiagrams of the extended Dynkin diagram of $\Phi$ that contain the affine node. This bijection is an isomorphism of posets with respect to the inclusions.
\end{pro}
\smallskip

\begin{rem}
\label{radice-minima}  Let $I \subseteq [n]$, $I\neq\emptyset$. 
\item{(1)} Consider $V_I$ and $\widehat\Phi^+
(\widehat \Pi\setminus \Pi_{\ov I})\setminus\Phi$ as partial order subsets of
the corresponding root posets with respect to the bases $\Pi$ and $\widehat
\Pi\setminus \Pi_{\ov I}$, respectively. The map from $V_I$ to $\widehat\Phi^+
(\widehat \Pi\setminus \Pi_{\ov I})\setminus\Phi$ sending $\beta$ to $-\beta + 
\delta$ (Lemma \ref{affine}) is an anti-isomorphism of posets. 
\item{(2)} Let $\eta_I$ be the minimal root in $V_I$ (Proposition \ref{principale}). Then $-\eta_I+\delta$ is the highest root of $\widehat\Pi\setminus \Pi_{\ov I}$, hence has positive coefficient in all roots in $\widehat \Pi\setminus \Pi_{\ov I}$ and, as observed in the the proof of Lemma \ref{affine}, the coefficient $c_0$ is $1$. It follows that $\ov I$ can be characterized in term of $\eta_I$: 
$$\ov I= \{i\in[n] \mid c_i(\eta_I) = m_i \}$$ 
(recall that, in our notation, $\theta= \sum m_i \alpha_i$). 
\end{rem}
\bigskip

In the following Proposition we characterize the non empty closed subsets of $[n]$.  
For any $\eta\in \Phi^+$, we set
$$I(\eta):=\{i\in [n]\mid c_i(\eta)=m_i\}.$$

\begin{pro}\label{ieoviuguale}
Let $I\subseteq[n]$, $I\neq\emptyset$. Then $I=\ov I$ if and only there exists 
$\eta\in\Phi^+$ such that $I=I(\eta)$.
\end{pro} 

\begin{proof}
As seen in the above remark, if $I=\ov I$, then $I=I(\eta_I)$. 
Conversely, let $\eta\in \Phi^+$.  
Then $\eta \in F_{I(\eta)}$, but $\eta \not\in F_{J}$ for all $J \supsetneq I(\eta)$, hence $I(\eta)$ is the maximum of $\{I \subseteq [n] \mid F_I= F_{I(\eta)} \}$.
By Proposition \ref{intervallo}, it follows that  $I(\eta)=\ov{I(\eta)}$.
\end{proof}
\smallskip

\begin{rem}
For all $w$ in the affine Weyl group $\widehat W$ of $\Phi$, let 
$N(w):=\{\alpha\in \widehat \Phi^+\mid  w^{-1}(\alpha)\in -\widehat \Phi^+\}$. It is is well known that $N(w)$ uniquely determines $w$ and that, for all $v, w\in \widehat W$, 
$N(vw)=\big(N(v) {\vartriangle}\, v(N^\pm(w))\big)\cap \widehat{\Phi}^+,$
where, for all $S \subseteq  \widehat \Phi^+$,
$S^\pm$ stands for $S\cup -S$, and $\vartriangle$ denotes the symmetric
difference. 
From this last relation, we easily obtain that, given a standard parabolic face $F_I$,   
$$-V_I+\delta=\widehat \Phi^+(\widehat\Pi\setminus \Pi_{\ov I})\setminus\widehat
\Phi^+(\Pi\setminus \Pi_{\ov I})=N(w_{0\ov I}\widehat w_{0\ov I} ),$$
where $\widehat w_{0\ov I} $ is the longest element in 
$\widehat W\la\widehat\Pi\setminus \Pi_{\ov I}\ra$ and  $w_{0\ov I}$ is the
longest element in $W\la\Pi\setminus\Pi_{\ov I}\ra$.
\par
It is a general fact that, if $V$ is an abelian dual order ideal of $\Phi^+$, 
then there exists an element $w_V$ in $\widehat W$ such that $-V+\delta=N(w_V)$.
In fact, the correspondence  $V\leftrightarrow w_V$ is a bijection between the abelian 
dual order ideals of $\Phi^+$ and the set of all $w$ in $\widehat W$ such that
$N(w) \subseteq -\Phi^++\delta$ (Peterson, see \cite{K}). 
\end{rem}
\smallskip

By Proposition \ref{intervallo},  the set of the $J$ giving the face $F_I$ is
an 
interval in the poset of the subsets of $[n]$ with minimum  $\partial I$ and maximum $\ov I$. 
We thus obtain, for every root system of rank $n$, a decomposition of
the Boolean algebra of rank $n$ as a disjoint union of intervals whose number
is the number of standard parabolic faces $+ 1$.
\par
Next we show that the maximum $\ov I$ yields the dimension and
the number of roots in the standard parabolic face $F_I$. We let
$E_{F_I}:= \gen_\real \{\alpha - \alpha' \mid \alpha,\alpha' \in F_I\}$ be
the vector subspace underlying the smallest affine subspace containing $F_I$.

\begin{pro}
\label{facceparabolichestandard}
\label{chiusi}
For all nonempty $I\subseteq[n]$, 
$$
E_{F_I}=\gen_\real \( ( \widehat \Pi \setminus \Pi_{I} )_0 \setminus
\{\alpha_0 \} \) =\gen_\real (\Pi\setminus \Pi_{\ov I}).\eqno{(1)}
$$ 
Hence, for all $I\subseteq[n]$,
$$
\dim F_I=\text{\it rk}\; \widehat\Phi_0(\widehat \Pi\setminus \Pi_I)-1=n-|\ov I|.
\eqno{(2)}
$$
Moreover, 
\item
$$
|V_I|= 
\frac{1}{2}\left(\Big|\widehat\Phi(\widehat\Pi\setminus \Pi_{ I})\Big| -
\Big|\Phi(\Pi\setminus \Pi_{I})\Big|\right)=
\frac{1}{2}\left(\Big|\widehat\Phi(\widehat\Pi\setminus \Pi_{\ov I})\Big| -
\Big|\Phi(\Pi\setminus \Pi_{\ov I})\Big|\right).
\eqno{(3)}
$$
\end{pro}

\begin{proof}
By Lemma \ref{affine}, $E_{F_I}$ is cointained in the subspace generated by the
differences of the roots in 
$\widehat\Phi_0(\widehat \Pi\setminus\Pi_I)\setminus \Phi$. Set  
$\beta:=\sum_{\alpha\in (\widehat \Pi\setminus\Pi_I)_0}\alpha$. Since
$\widehat\Phi_0(\widehat \Pi\setminus\Pi_I)$ is irreducible, $\beta$  is a
root and there exists a chain  $\beta_0=\alpha_0, \beta_1, \dots,
\beta_s=\beta$  in $\widehat\Phi(\widehat \Pi\setminus \Pi_I)\setminus \Phi$,
such that $\{\beta_0, \beta_i-\beta_{i-1}\mid i=1, \dots, s\}=(\widehat
\Pi\setminus\Pi_I)_0$. 
This proves (1).
\par
Equation (2) is obvious for $I=\emptyset$, and follows by Equation (1) for
$I\neq
\emptyset$, since $\dim F_I=\dim E_{F_{I}}$. 
\par
Equation (3) follows directly from Lemma \ref{affine} and from the definition of $\ov I$.
\end{proof}
\smallskip

We recall that, for every nonempty $I \subseteq [n]$, $(V_I)_\ell$ and 
$(V_I)_s$ denote, respectively, the long and the short roots in the standard
parabolic face $F_I$.
\smallskip

\begin{lem} \label{VIell}
\label{transitivo}
The parabolic subgroup 
$W\la\Pi\setminus\Pi_{\ov I}\ra$ acts
transitively on $(V_I)_\ell$, and on $(V_I)_s$.
\end{lem}

\begin{proof}
Clearly, $\theta\in (V_I)_\ell$. We prove that any $\gamma\in(V_I)_\ell$ can 
be transformed into $\theta$ by some $w$ in $W\la\Pi\setminus\Pi_{\ov I}\ra$. By contradiction, let $\gamma$ be a counterexample of maximal height. Then there exists some $\alpha\in\Pi$ such that $(\gamma, \alpha)<0$ since $\theta$ is the unique long root in the closure of the fundamental Weyl chamber, and thus $s_\alpha(\gamma)=\gamma+c\alpha$ with $c>0$. 
By definition of $V_I$, we have that $\alpha\in\Pi\setminus\Pi_I$ and $\gamma+c\alpha\in V_I$. Indeed, by  Proposition \ref{intervallo}, $\alpha\in\Pi\setminus\Pi_{\ov I}$.
By the maximality of $\text{\it ht}(\gamma)$, there exists $w\in W\la\Pi\setminus\Pi_{\ov I}\ra$ such that  $\theta=w(\gamma+c\alpha)=w s_\alpha(\gamma)$: a contradiction, since  $w s_\alpha \in W\la\Pi\setminus\Pi_{\ov I}\ra$.
\par
Now assume $(V_I)_s\neq\emptyset$. Then the highest short root $\theta_s$ belongs to it, since $V_I$ is a dual order ideal.  Since 
$\theta_s$ is the unique dominant short root, we can prove that any other short root $\beta$ in $V_I$ can be transformed into $\theta_s$ by some $w$ in  $W\la\Pi\setminus\Pi_{\ov I}\ra$, with the same argument used for the long roots.
\end{proof}
\smallskip

In the next Proposition we see that 
while the maximum $\ov I$ yields the dimension and the number of roots in 
the standard parabolic face $F_I$ (Proposition \ref{chiusi}), the minimum
$\partial I$ yields the stabilizer of $F_I$. 

\begin{pro}
\label{stabilizzatore}
For each $I\subseteq [n]$, let 
$\text{\it Stab}_W F_I=\{w\in W\mid w (F_I)=F_I\}$. Then 
$$
\text{\it Stab}_W F_I=W\la\Pi\setminus \Pi_{\partial I}\ra =
W\la \Pi\setminus \Pi_{\ov I} \ra\times 
W\la\Pi_{\ov I}\setminus {\Pi_{\partial I}}\ra.$$
Moreover, 
$W\la\Pi_{\ov I}\setminus {\Pi_{\partial I}}\ra$ fixes the face $F_I$ pointwise, while the action of each $w$ in $W\la \Pi\setminus \Pi_{\ov I}\ra$ on $F_I$ is nontrivial, unless $w=e$, the identity of $W$. 
\end{pro}

\begin{proof}
It is clear that $\text{\it Stab}_W F_I=\text{\it Stab}_W V_I=
\{w\in W\mid w (V_I)=V_I\},$ therefore  
it is immediate that $W\la\Pi\setminus \Pi_I\ra\subseteq\text{\it Stab}_W F_I$. 
This should happen for all $J$ such that $F_J=F_I$,  in particular for $\partial I$, 
therefore $W\la\Pi\setminus \Pi_{\partial I}\ra \subseteq \text{\it Stab}_W F_I$. 
\par 
Now, assume by contradiction that 
$\text{\it Stab}_W F_I\setminus W\la\Pi\setminus \Pi_{\partial I}\ra\neq\emptyset$ and let $w$ be an element of minimal length in 
$\text{\it Stab}_W F_I\setminus W\la\Pi \setminus \Pi_{\partial I}\ra$. Then  
$D_r(w)\subseteq \partial I$, therefore there 
exists $i \in \partial I$ such that $w(\alpha_i)\in-\Phi^+$. 
Let $\eta=\min V_I$ and $\alpha_i=\alpha$ for short. Then, 
by the proof of Proposition \ref{laminimaelunga} and by definition of $\partial I$, $(\eta, \alpha)>0$. 
It follows, by Proposition \ref{standard}, that $\eta-\alpha\in \Phi\cup\{0\}$: 
on the other hand, $\eta \neq \alpha$ since $\eta \in F_I$ while $\alpha \notin
F_I$, 
because $w(\alpha)\notin F_I$ 
and $w$ stabilizes $F_I$.
By assumption, $w(\eta)\in V_I$, in particular,  $w(\eta)\in
\Phi^+$, therefore
$w(\eta-\alpha)=w(\eta)-w(\alpha)$ is a root which is a sum of positive roots, 
one of which is in $F_I$. 
Hence $w(\eta-\alpha) \in F_I$, 
which is a contradiction since $\eta-\alpha \notin F_I$.
\par
Since the diagram of $\Pi\setminus \Pi_{\partial I}$ is the disjoint union of
the diagrams of  
$\Pi\setminus \Pi_{\ov I}$ and $\Pi_{\ov I}\setminus \Pi_{\partial I}$,  $W\la\Pi\setminus \Pi_{\partial I}\ra$ is the direct product
$W\la \Pi\setminus \Pi_{\ov I} \ra\times 
W\la\Pi_{\ov I}\setminus {\Pi_{\partial I}}\ra.$
By Lemma \ref{affine}, all elements in $V_I$ are orthogonal to all roots in 
$\Pi_{\ov I}\setminus {\Pi_{\partial I}}$, therefore $W\la\Pi_{\ov I}\setminus {\Pi_{\partial I}}\ra$ fixes $F_I$ pointwise.
Finally, each $w\in W\la \Pi\setminus \Pi_{\ov I}\ra$ is nontrivial on $F_I$,
unless $w=e$,  since $W\la \Pi\setminus \Pi_{\ov I}\ra$ 
acts faithfully on $\gen_\real(\Pi\setminus \Pi_{\ov I})=E_{F_I}$.
\end{proof}
\smallskip

\begin{rem}\label{cortesuperflue}
It is easy to see that the root polytope $\mc P_\Phi$ is indeed the convex hull
of the long roots. 
In fact, if $\Phi$ is not simply laced, since $\Phi$ is irreducible, any short
root is contained in a rank 2 non-simply laced subsystem.  And it is immediate to
check that in such a subsystem a short root can be obtained as a convex linear
combination of two long roots.  
\par
In particular we have that:
$$
\mathcal P_{B_3}=\mathcal P_{A_3},\quad 
\mathcal P_{B_n}=\mathcal P_{D_n}\text{  for $n \geq 4$},\quad 
\mathcal P_{F_4}=\mathcal P_{D_4},\quad
\text{and}\quad 
\mathcal P_{G_2}=\mathcal P_{A_2}.
$$
We explicitly notice that $\mathcal P_{C_n}$ is the cross-polytope for all $n
\geq 2$ (octahedron for $n=3$), and that $\mathcal P_{A_n}$ and $\mathcal
P_{B_n}=\mathcal P_{D_n}$, for $n \geq 4$, are distinct $n$-dimensional
generalizations of the cuboctahedron $\mathcal P_{A_3}=\mathcal P_{B_3}$ (see
\cite{C}).
\end{rem}
\smallskip

Thus, the vertices of $\mc P_\Phi$ are long roots, and are clearly all the long roots, since $W$ acts on $\mc P_\Phi$ and is transitive on the long roots.
\par
In particular, the set of vertices in the standard parabolic face $F_I$ is the subset of long roots of $V_I$.  By Lemma \ref{transitivo}, this set is the orbit of $\theta$ under the action of $W \la\Pi\setminus \Pi_{\ov I} \ra$.
Since $\theta$ is dominant, its stabilizer in $W \la\Pi\setminus \Pi_{\ov I} \ra$ is the parabolic subgroup  generated by the simple reflections through the roots in
$(\Pi\setminus \Pi_{\ov I})\cap \theta^\perp$. Therefore,  we obtain the following corollary. 
\smallskip

\begin{cor}\label{numerovertici}
The number of vertices of the standard parabolic face $F_I$ is 
$[W\la\Pi\setminus \Pi_{\ov I}\ra : W\la \(\Pi\setminus\Pi_{\ov I}\)\cap
\theta^{\perp}\ra]$.
\end{cor}
\smallskip

 Clearly, $\(\Pi\setminus \Pi_{\ov I}\)\cap\theta^\perp$ is the subset of  
the roots in $\Pi\setminus \Pi_{\ov I}$ that are not adjacent to $\alpha_0$ in
the extended Dynkin diagram.  
\smallskip

The following proposition, based on Lemma \ref{transitivo},  is a characterization of the roots of type $\min V_I$, for some nonempty $I\subseteq [n]$.

\begin{pro}
\label{carat-min}
Let $\eta\in \Phi^+$. Then
$\eta=\min V_{I(\eta)}$ if and only  $\eta$ is long and $(\eta, \alpha_i)\leq 0$ for all $i\in [n]\setminus{I(\eta)}$. 
\par
In particular, the set of the standard parabolic faces of $\mathcal P_\Phi$ is in bijection with the subset
$\{\eta\in \Phi_\ell^+\mid (\eta, \alpha_i)\leq 0, \text{ for all } i\in  [n]\setminus I(\eta)\}$ of $\Phi^+$.
\end{pro}

\begin{proof}
First, we assume that  $\eta=\eta_{I(\eta)}$.  Then $\eta\in \Phi_\ell$ by Proposition \ref{principale}. By Lemma \ref{affine} and the proof of  Proposition \ref{principale}, $-\eta+\delta$ is the highest root of $\Phi\(\{\alpha_0\}\cup  \(\Pi\setminus \Pi_{I}\)\)$, in particular $-\eta+\delta$ is dominant for this root system, which implies that $(\eta, \alpha_i)\leq 0$ for all $i\in \Pi\setminus \Pi_{I(\eta)}$. 
\par
Next, we assume $(\eta, \alpha_i)\leq 0$ for all $i\in [n]\setminus{I(\eta)}$ and that $\eta\in \Phi_\ell$.
The first condition implies that $w(\eta)\geq \eta$  for all $w\in W\la \Pi\setminus \Pi_{{I(\eta)}}\ra$: this is easily seen by induction on the length of $w$.
Indeed, let $w= s_{\alpha_{i_1}} \ldots s_{\alpha_{i_k}}$ be a reduced expression. Then $w(\eta)=  s_{\alpha_{i_1}} \ldots s_{\alpha_{i_{k-1}}} (\eta + c \alpha_{i_k})$, with $c \geq 0$. By induction, $s_{\alpha_{i_1}} \ldots s_{\alpha_{i_{k-1}}} (\eta) \geq \eta$ and, by the properties of reduced expressions, $s_{\alpha_{i_1}} \ldots s_{\alpha_{i_{k-1}}} (\alpha_{i_k}) > 0$. The claim follows. 
 Hence, $\eta$ is a minimal element in its orbit under the action of  $W\la \Pi\setminus \Pi_{{I(\eta)}}\ra$. 
Since $\eta$ is long, by  Lemma \ref{transitivo}, this orbit is equal to $(V_{I(\eta)})_\ell$, and by Proposition \ref{principale} $\min V_{I(\eta)}=\min (V_{I(\eta)})_\ell$. 
 It follows that $\eta=\min V_{I(\eta)}$.  
\end{proof}

\section{Coordinate faces}

Recall that we call coordinate faces the standard parabolic faces of type
$F_i$,
$i\in [n]$. 
It is clear, from the definition, that the standard parabolic facets of
$\mc P_\Phi$  must be of this type, but not all coordinate faces are
facets. Indeed, two coordinate faces can be one included in the other: the
resulting partial order structure on the set of coordinate faces (or,
equivalently by Proposition \ref{faccecoordinatedistinte}, on $[n]$) 
has a simple uniform description in terms of the Dynkin diagram (see Remark
\ref{ordine}).
\par
By the results in the previous section, we know that, for all $I\subseteq
[n]$, the standard parabolic face $F_I$ coincides with $F_{\partial I}$:
$\partial I$ is a small subset of $[n]$, being
the set of roots adjacent and exterior to some irreducible subdiagram of the
extended Dynkin diagram. In fact, it is easy to check that $\partial I$ has at
most $3$ elements (more precisely, exactly $3$ elements only in a special case
occurring in type $D$ and at most $2$ elements otherwise). This means
that every standard parabolic face $F_I$ can be obtained as the intersection of
$3$ or (almost always) fewer
coordinate faces.  
We dedicate this section to a deeper analysis of the coordinate
faces.

\begin{pro}  
 \label{faccecoordinatedistinte}
All coordinate faces are distinct.
\end{pro}

\begin{proof}
The proof follows by Lemma \ref{affine} since, for all $k,h \in [n]$, the connected components of $\alpha_0$ in the Dynkin diagrams of $\widehat \Pi \setminus \{\alpha_k\}$ and  $\widehat \Pi \setminus \{\alpha_h\}$ cannot coincide if $k\neq h$. 
\end{proof}
\smallskip

Proposition \ref{faccecoordinatedistinte} follows also by the following
argument. Let $\eta_i$ be the minimal root in the face $F_i$, whose existence
is established by Proposition \ref{principale}. Then the only simple root that
can have positive scalar product with $\eta_i$ is $\alpha_i$. Indeed, if
$\alpha_j$, $j\neq i$, were a simple root having positive scalar product with
$\eta_i$, then $\eta_i - \alpha_j$ would be a root in $F_i$ by Proposition
\ref{standard}, which is impossible since $\eta_i$ is the minimum. 
On the other hand, $\eta_i$ is a positive root and hence cannot have
non-positive scalar product with all simple roots. Thus $\eta_h \neq \eta_k$ if
$h \neq k$, and the faces $F_h$ and $F_k$ cannot coincide.
\par
Another property of the coordinate face $F_i$, implying Proposition \ref{faccecoordinatedistinte}, is that the barycenter of $V_i$
is parallel to the $i$-th fundamental coweight $\breve{\omega}_i$.
This follows by the following general lemma, which states that every 
$\alpha$-string is centered on a vector orthogonal to $\alpha$.

\begin{lem}
\label{centrato}
 Let $\alpha$ and  $\beta$ be non-proportional roots in $\Phi$, and let 
$\alpha$-$(\beta)$ be the $\alpha$-string through $\beta$.
Set $\mu := \sum_{\gamma\in \textrm{$\alpha$-$(\beta)$}} \gamma$. 
Then
$$(\mu, \alpha)=0.$$
\end{lem}

\begin{proof}
We may suppose that $\beta$ is the origin of its $\alpha$-string. Then, by
Proposition \ref{standard2}, 
$\alpha$-$(\beta) =\left\{ \beta + j \alpha \mid j=0,1, \ldots, -
\frac{2(\beta, \alpha)}{(\alpha, \alpha)} \right\}$.
The middle vector $\beta -  \frac{(\beta, \alpha)}{(\alpha, \alpha)} \alpha $ is
orthogonal to $\alpha$. 
\end{proof}
\smallskip

\begin{pro}
\label{bari-coor}
The barycenter of the roots in the $i$-th coordinate face $F_i$ is parallel to the $i$-th fundamental coweight: 
$$\sum_{\alpha \in V_i} \alpha = \frac{m_i |V_i|}{(\breve{\omega}_i,\breve{\omega}_i)}\,  \breve{\omega}_i. $$
\end{pro}

\begin{proof}
 By the definition of $V_i$, if $\alpha \in V_i$ and $\alpha \pm \alpha_j \in
\Phi$ for a certain $j \neq i$, 
then $\alpha \pm \alpha_j \in V_i$.
Hence $V_i$ is a union of $\alpha_j$-string, for all $j \neq i$. By Lemma
\ref{centrato},   $( \sum_{\alpha \in V_i} \alpha, \alpha_j) = 0$, 
for all $j \neq i$. Hence   $\sum_{\alpha \in V_i} \alpha$ is a multiple of
$\breve{\omega}_i$. 
Since $(\alpha, \breve{\omega}_i) = m_i$ for all 
$\alpha \in V_i$, we get the assertion.
\end{proof}
\smallskip

Notice that, in the proof of Proposition \ref{bari-coor}, we use the fact that the set
of roots $V_i$ in the coordinate face $F_i$ is union of $\alpha_j$-string, for
all $j \neq i$. Consequently,  the set of roots $V_I$ in the 
standard parabolic
face $F_I$ is a union
of $\alpha_j$-string, for all $j\notin I$. We obtain the following result on the
barycenter of the roots in a standard parabolic face.

\begin{pro}
\label{baryst}
Let $I\subseteq [n]$. The barycenter of the set of roots in the standard
parabolic face $F_I$ is in 
the cone generated by the coweights $\breve{\omega}_i$, $i\in \partial I$. 
\end{pro}

\begin{proof}
By Proposition \ref{intervallo}, we may clearly assume that $I=\partial I$. 
Since $V_I$ is union of $\alpha_j$-string, for all $j\notin I$, the barycenter 
$\sum_{\alpha \in V_I}\alpha$ of $F_I$ is orthogonal to $\alpha_j$, for all
$j\notin I$, by Lemma \ref{centrato}. Hence it is in the span of the coweights
$\breve{\omega}_i$, $i\in I$. Moreover, by Proposition \ref{standard}, $(\alpha,
\alpha_i) \geq 0$ for all $\alpha \in V_I$ and $i \in I$, since
$\alpha+\alpha_i$ cannot be a root. Hence the barycenter has nonnegative scalar
product with all $\alpha_i$, $i \in I$, and we have the assertion.
\end{proof}
\smallskip

The following result provides several conditions equivalent to the fact that 
the coordinate face $F_i$ is a facet of the root polytope $\mathcal P_{\Phi}$.

\begin{thm}
\label{faccette}
Let $i \in [n]$. The following are equivalent.
\begin{enumerate}
 \item The coordinate face $F_i$ is a facet of $\mathcal P_{\Phi}$.
\label{1}

\item The standard parabolic subsystem $\widehat \Phi(\widehat\Pi\setminus \{\alpha_i\})$ is irreducible.
\label{0}

\item The minimal root $\eta_i$ of $V_i$ satisfies $(\eta_i, \breve{\omega}_j)
\neq m_j$, for all $j\neq i$.
\label{2}

\item $\ov{ \{i\} }= \{i\}$.
\label{6}

\item For all $j \neq i$, there exists $\alpha \in V_i$ such that $(\alpha, \breve{\omega}_j) \neq m_j$.
\label{3}

\item The set $V_i$ contains an $\alpha_j$-string which is non-trivial (i.e. of cardinality $> 1$), for all $j \neq i$.
\label{4}

\item $F_i$ is maximal among the coordinate faces $\{F_j \mid j=1, \ldots, n\}$.
\label{5}
\end{enumerate}
\end{thm}

\begin{proof}
(\ref{6}) is equivalent to (\ref{1}) by Proposition
\ref{facceparabolichestandard}, to (\ref{0}) by definition, and to (\ref{2}) by
Remark \ref{radice-minima}.
\par
We prove \ref{2} $\Rightarrow$ \ref{3}
$\Rightarrow$ \ref{4} $\Rightarrow$ \ref{1} and then 
\ref{1} $\Rightarrow$ \ref{5} $\Rightarrow$ \ref{4}. 
\par
The assertion \ref{2}
$\Rightarrow$ \ref{3} is trivial.
Fix $j \neq i$ and suppose there exists $\alpha \in V_i$ such that $(\alpha,
\breve{\omega}_j) \neq m_j$. 
Take a chain $\alpha=\gamma_0 \lhd \gamma_1 \lhd \cdots \lhd \gamma_t= \theta$
in the root poset. 
Since $(\alpha, \breve{\omega}_j) \neq m_j$ and the root poset is ranked, there
exists $r$ such that $\gamma_{r+1}= \gamma_r + \alpha_j$. 
Hence the $\alpha_j$-string through $\gamma_r$ 
is non-trivial and we have \ref{3} $\Rightarrow$ \ref{4}. Clearly \ref{4}
$\Rightarrow$ \ref{1} 
since the difference of two consecutive roots
in an $\alpha_j$-string is $\alpha_j$.
\par
It is trivial that  \ref{1} $\Rightarrow$ \ref{5}. Let us 
prove \ref{5} $\Rightarrow$ \ref{4} by contradiction. So assume that there
exists 
$j \neq i$ such that all $\alpha_j$-strings in $V_i$ are trivial.  
For  every $\alpha \in V_i$, the difference of two consecutive roots in any
chain in the root poset from $\alpha$ to $\theta$ cannot be $\alpha_j$. This
means that $(\alpha, \breve{\omega}_j) = m_j$ for all $\alpha \in V_i$. 
Hence $V_j \supseteq V_i$. Since all the coordinate faces $F_k$ are distinct by
Proposition \ref{faccecoordinatedistinte}, $F_i$ is not maximal.  
\end{proof}
\smallskip

Note that, if $m_i = 1$, then $\alpha_i$ coincides with the minimal root
$\eta_i$ of $V_i$, and hence $F_i$ is a facet. So, for example, 
in type $A$, all coordinate faces $F_i$ are facets. On the other hand, 
in $B_n$ ($n\geq 3$), $C_n$ ($n\geq 2$), and $D_n$ ($n\geq 4$) there are,
respectively, only 2, 1, and 3 coordinate
facets: $F_1$ and $F_n$ in $B_n$, $F_n$ in $C_n$, and $F_1$, $F_{n-1}$ and
$F_n$ in $D_n$.

\begin{rem}
\label{ordine}
By the results in the previous section, the partial order on the set of the
coordinate faces can be characterized as
 follows. Let  $i,j \in [n]$. Then the following are equivalent:
\begin{enumerate}
 \item $F_i \subseteq
F_j$,
\item $\ov{\{i\}} \supseteq \ov{\{j\}}$,
\item in the affine Dynkin diagram of $\Phi$, every minimal path  from
$\alpha_j$ to $\alpha_0$ contains $\alpha_i$. 
\end{enumerate}
\end{rem}

The Hasse diagrams of the posets of the coordinate faces under inclusion are
depicted in Tables \ref{Thasse} and \ref{ThasseEFG} for all types. The
numerations of the simple roots follow those in \cite{Bou}.

\begin{table}[b]
\caption{Hasse diagrams of the coordinate faces set for the classical Weyl
groups}
\label{Thasse}
\centering
\begin{tabular}{l| c} 
\\
\begin{tikzpicture}
\node{$A_n$};
\end{tikzpicture} 
&
\begin{tikzpicture}[scale=.8]
{\foreach \x in {0, 1, 2, 5, 6}
\draw[fill=white]{(\x,0) circle(3pt)};}
{\foreach \x in {3, 3.5, 4}
\draw[fill=black]{(\x,0) circle(1pt)};}
\node[below right] at (-.4,-.1)
{$\ssc F_1$};
\node[below right] at (.6,-.1)
{$\ssc F_2$};
\node[below right] at (1.6,-.1)
{$\ssc F_3$};
\node[below right] at (4.6,-.1)
{$\ssc F_{n-1}$};
\node[below right] at (5.6,-.1)
{$\ssc F_{n}$};
\end{tikzpicture}

\\ \hline\\

\begin{tikzpicture}
\node{$B_n$};
\end{tikzpicture}
&
\begin{tikzpicture}[scale=.8]
{\foreach \x in {0, 1, 4, 5}
\draw[fill=white]{(0,\x) circle(3pt)};}
\draw[fill=white]{(-2,2.5) circle(3pt)};
{\foreach \x in {2, 2.5, 3}
\draw[fill=black]{(0,\x) circle(1pt)};}
\draw{(0,0.1)--(0,0.9) (0,4.1)--(0,4.9)};
\draw{(-0.1,0.1)--(-1.9,2.4)};
\node[below right] at (-2.8,2.8)
{$\ssc F_1$};
\node[below right] at (.1,.3)
{$\ssc F_2$};
\node[below right] at (.1,1.3)
{$\ssc F_3$};
\node[below right] at (.1, 4.3)
{$\ssc F_{n-1}$};
\node[below right] at (.1, 5.3)
{$\ssc F_{n}$};
\end{tikzpicture}

\\ \hline\\

\begin{tikzpicture}
\node{$C_n$};
\end{tikzpicture}
&
\begin{tikzpicture}[scale=.8]
{\foreach \x in {0, 1, 4, 5}
\draw[fill=white]{(0,\x) circle(3pt)};}
{\foreach \x in {2, 2.5, 3}
\draw[fill=black]{(0,\x) circle(1pt)};}
\draw{(0,0.1)--(0,0.9) (0,4.1)--(0,4.9)};
\node[below right] at (.1,.3)
{$\ssc F_1$};
\node[below right] at (.1,1.3)
{$\ssc F_2$};
\node[below right] at (.1, 4.3)
{$\ssc F_{n-1}$};
\node[below right] at (.1, 5.3)
{$\ssc F_{n}$};
\end{tikzpicture}

\\ \hline\\
\begin{tikzpicture}
\node{$D_n$};
\end{tikzpicture}
&
\begin{tikzpicture}[scale=.8]
{\foreach \x in {0, 1, 4, 5}
\draw[fill=white]{(0,\x) circle(3pt)};}
\draw[fill=white]{(-0.7,6) circle(3pt)};
\draw[fill=white]{(0.7,6) circle(3pt)};
{\foreach \x in {2, 2.5, 3}
\draw[fill=black]{(0,\x) circle(1pt)};}
\draw{(0,0.1)--(0,0.9) (0,4.1)--(0,4.9) (-0.05,5.1)--(-0.65, 5.9)
(0.05,5.1)--(0.65,5.9)};
\node[below right] at (-1.6,6.3)
{$\ssc F_n$};
\node[below right] at (0.8,6.3)
{$\ssc F_{n-1}$};
\node[below right] at (-2.8,2.8)
{$\ssc F_1$};
\node[below right] at (.1,.3)
{$\ssc F_2$};
\node[below right] at (.1,1.3)
{$\ssc F_3$};
\node[below right] at (.1, 4.3)
{$\ssc F_{n-3}$};
\node[below right] at (.1, 5.3)
{$\ssc F_{n-2}$};
\draw{(-0.1,0.1)--(-1.9,2.4)};
\draw[fill=white]{(-2,2.5) circle(3pt)};
\end{tikzpicture}

\\ \hline\\
\end{tabular}
\end{table}

\begin{table}[]
\caption{Hasse diagrams of the coordinate faces set for the exceptional Weyl
groups}
\label{ThasseEFG}
\centering
\begin{tabular}{l| c} 

\begin{tikzpicture}[scale=.6]
\node at (0,0) {$E_6$};
\end{tikzpicture}
&
\begin{tikzpicture}[scale=.8]
{\foreach \x in {4, 5}
\draw[fill=white]{(0,\x) circle(3pt)};}
\draw[fill=white]{(-0.71,6) circle(3pt)};
\draw[fill=white]{(0.71,6) circle(3pt)};
\draw[fill=white]{(1.45,6.97) circle(3pt)};
\draw[fill=white]{(-1.45,6.97) circle(3pt)};
\draw{ (0,4.1)--(0,4.9) (-0.05,5.1)--(-0.65, 5.9)
(0.05,5.1)--(0.65,5.9)  (-0.79,6.09)--(-1.39, 6.89)
(0.79,6.09)--(1.39, 6.89)    };
\node[below right] at (-1.6,6.3)
{$\ssc F_3$};
\node[below right] at (0.8,6.3)
{$\ssc F_{5}$};
\node[below right] at (.1, 4.3)
{$\ssc F_{2}$};
\node[below right] at (.1, 5.3)
{$\ssc F_{4}$};
\node[below right] at (1.55, 7.25)
{$\ssc F_{6}$};
\node[below right] at (-2.3, 7.25)
{$\ssc F_{1}$};
\end{tikzpicture}

\\ \hline\\

\begin{tikzpicture}
\node{$E_7$};
\end{tikzpicture}
&
\begin{tikzpicture}[scale=.8]
{\foreach \x in {3, 4, 5, 6, 7, 8}
\draw[fill=white]{(0,\x) circle(3pt)};}
\draw[fill=white]{(-1.45,6.97) circle(3pt)};
\draw{(0,3.1)--(0,3.9) (0,4.1)--(0,4.9) (0,5.1)--(0,5.9)
(0,6.1)--(0,6.9) (0,7.1)--(0,7.9)
(-0.05,5.1)--(-1.39, 6.89)};
\node[below right] at (0.1,6.3)
{$\ssc F_{5}$};
\node[below right] at (.1, 4.3)
{$\ssc F_{3}$};
\node[below right] at (.1, 5.3)
{$\ssc F_{4}$};
\node[below right] at (.1, 7.3)
{$\ssc F_{6}$};
\node[below right] at (.1, 8.3)
{$\ssc F_{7}$};
\node[below right] at (-2.3, 7.25)
{$\ssc F_{2}$};
\node[below right] at (.1, 3.3)
{$\ssc F_{1}$};
\end{tikzpicture}

\\ \hline\\

\begin{tikzpicture}
\node{$E_8$};
\end{tikzpicture}
&
\begin{tikzpicture}[scale=.8]
{\foreach \x in {2, 3, 4, 5, 6, 7, 8}
\draw[fill=white]{(0,\x) circle(3pt)};}
\draw[fill=white]{(-1.45,7.97) circle(3pt)};
\draw{(0,2.1)--(0,2.9) (0,3.1)--(0,3.9) (0,4.1)--(0,4.9) (0,5.1)--(0,5.9)
(0,6.1)--(0,6.9) (0,7.1)--(0,7.9)
(-0.05,6.1)--(-1.39, 7.89)};
\node[below right] at (0.1,6.3)
{$\ssc F_{4}$};
\node[below right] at (.1, 4.3)
{$\ssc F_{6}$};
\node[below right] at (.1, 5.3)
{$\ssc F_{5}$};
\node[below right] at (.1, 7.3)
{$\ssc F_{3}$};
\node[below right] at (.1, 8.3)
{$\ssc F_{1}$};
\node[below right] at (-2.3, 8.25)
{$\ssc F_{2}$};
\node[below right] at (.1, 3.3)
{$\ssc F_{7}$};
\node[below right] at (.1, 2.3)
{$\ssc F_{8}$};
\end{tikzpicture}

\\ \hline\\

\begin{tikzpicture}
\node{$F_4$};
\end{tikzpicture}
&
\begin{tikzpicture}[scale=.8]
{\foreach \x in {2, 3, 4, 5}
\draw[fill=white]{(0,\x) circle(3pt)};}
\draw{(0,2.1)--(0,2.9) (0,3.1)--(0,3.9) (0,4.1)--(0,4.9) };
\node[below right] at (.1, 4.3)
{$\ssc F_{3}$};
\node[below right] at (.1, 5.3)
{$\ssc F_{4}$};
\node[below right] at (.1, 3.3)
{$\ssc F_{2}$};
\node[below right] at (.1, 2.3)
{$\ssc F_{1}$};
\end{tikzpicture}

\\ \hline\\

\begin{tikzpicture}
\node{$G_{2}$};
\end{tikzpicture}
&
\begin{tikzpicture}[scale=.8]
{\foreach \x in {2, 3}
\draw[fill=white]{(0,\x) circle(3pt)};}
\draw{(0,2.1)--(0,2.9)  };
\node[below right] at (.1, 3.3)
{$\ssc F_{1}$};
\node[below right] at (.1, 2.3)
{$\ssc F_{2}$};
\end{tikzpicture}
\\ \hline\\
\end{tabular}
\end{table}

\section{Uniform description of $\mathcal P_{\Phi}$}

Clearly, the Weyl group $W$ acts on the set of the faces of $\mathcal
P_{\Phi}$.
 We say that a face $F$ is a \emph{parabolic face} if it is transformed into a
standard parabolic face by an element in $W$.  

\begin{pro}
\label{nonconiugate}
Two distinct standard parabolic faces 
cannot be transformed into one another by elements in the Weyl group $W$. 
\end{pro}

\begin{proof}
The barycenter of every standard parabolic face is in the closure of the 
fundamental Weyl chamber, by Proposition 
\ref{baryst}. Since the closure of the fundamental chamber is a fundamental
domain for the action of $W$, the barycenters of two standard 
parabolic faces in the same $W$-orbit must coincide. Since distinct faces have
distinct barycenters, we get the assertion.
\end{proof}
\smallskip

Given an arbitrary face $F$ of $\mathcal P_{\Phi}$, we let  $V_F := \Phi \cap F$
be the set of the roots in $F$ (so, for all $I \subseteq [n]$, $V_{F_I} = V_I$),
and $E_F:= \gen_\real \{\alpha - \alpha' \mid \alpha,\alpha' \in F\}$ be
the vector subspace underlying the smallest affine subspace containing $F$.
\par
The first main result of this section, Theorem \ref{nonspezzato}, is 
that for any face $F$, $V_F$ cannot be partitioned into two nonempty mutually
orthogonal sets. 
We remark that
this is false if we consider the orbit of $\theta$ instead
of the set of all roots, as we can see in $C_2$.

\begin{thm}
\label{nonspezzato}
Let $F$ be a face of $\mathcal P_{\Phi}$. The set $V_F$ of the roots in $F$ is
not the union of two non-trivial orthogonal subsets.
\end{thm}

\begin{proof}
By contradiction, suppose $V_F = V_1 \cup V_2$, where $V_1$ and $V_2$ are
non-empty sets such that every root in $V_1$ is orthogonal to every  root in
$V_2$. 
We first prove that, under this assumption, the following holds:
\par\noindent
(1) {\it there is no $\alpha\in\Phi$ that is
simultaneously not orthogonal to $V_1$ and $V_2$. }
\par
Assume, by contradiction, that there exist $\alpha \in \Phi$, $\beta_i \in
V_i$, $i=1,2$, such that $(\alpha,\beta_i)\neq 0$, $i=1,2$. We
may assume $(\alpha,\beta_1)<0$, hence $\alpha+\beta_1\in \Phi$.
Let $f$ be a linear functional such that 
$F=\{x\in \mathcal P_\Phi\mid f(x)=m_f\}$. We may assume $m_f
>0$. By the symmetry of $\mathcal P_\Phi$, $f$ 
takes values between $-m_f$ and $m_f$ on $\Phi$.
If $f(\alpha)=0$,  then $\alpha+\beta_1\in F$, and since 
$\alpha+ \beta_1\not\perp \beta_2$, we obtain that $\alpha+\beta_1\in V_2$. It
follows that  $\alpha+\beta_1\perp \beta_1$, hence $s_{\beta_1}(\alpha)=
s_{\beta_1}(\alpha + \beta_1 - \beta_1)= \alpha + \beta_1
- s_{\beta_1} (\beta_1)= \alpha + 2 \beta_1$. 
But $f(\alpha + 2 \beta_1)= 2 m_f$: a contradiction. Thus, 
$f(\alpha)\neq 0$. 
Since $\alpha$ and $\alpha+\beta_1$ are roots, we have that $-m_f\leq
f(\alpha)<0$. This implies that $f(\beta-\alpha)>m_f$, for all $\beta\in V_F$,  
hence  that $\beta-\alpha$ cannot be a root for all $\beta\in V_F$. Therefore,
 $(\alpha, \beta_2)<0$,
and hence $\alpha+\beta_1+ \beta_2$ is a root.  This forces $f(\alpha)=-m_f$,
hence $-\alpha\in V_F=V_1\cup V_2$: a contradiction, since $-\alpha$ is not
orthogonal to $V_1$ nor to $V_2$.  Thus  (1) is proven. 
\par
Now, let $\Phi_F := \Phi \cap \gen_\real V_F$,  
$\Phi_1 := \Phi \cap \gen_\real V_1$, and 
$\Phi_2 := \Phi \cap \gen_\real V_2$.
We prove that $\Phi_F= \Phi_1 \cup \Phi_2$.
Suppose $\alpha \in \Phi_F \setminus (\Phi_1 \cup \Phi_2)$. Being 
$\gen_\real V_F= \gen_\real V_1 + \gen_\real V_2$,
$\gen_\real V_1$ 
is the orthogonal complement of $\gen_\real V_2$ in $\gen_\real V_F$, and
vice-versa. Hence, there exists $\beta_i \in V_i$ such that 
$(\alpha,\beta_i)\neq 0 $, for $i=1,2$: this is can not happen, by (1), hence
$\Phi_F= \Phi_1 \cup \Phi_2$.
\par
Finally, we are going to prove that we can find $\alpha\in \Phi$ which is
simultaneously not orthogonal to $V_1$ and $V_2$: this will conclude the proof,
since contradicts (1).     
By the previous step, and Proposition \ref{standard3}, we can find three
subsets 
$\widetilde{\Pi}, \widetilde{\Pi}_1, \widetilde{\Pi}_2 \subseteq \Phi$, 
which are respectively bases 
of $\Phi$, $\Phi_1$ and $ \Phi_2$, and such that $\widetilde{\Pi}_1,
\widetilde{\Pi}_2 \subseteq \widetilde{\Pi}$.
Consider any pair of roots $\gamma_i\in\Pi_i$, $i=1, 2$. Since $\Phi$ is
irreducible, there is a simple path $L$  connecting $\gamma_1$ and $\gamma_2$ on
the Dynkin graph of $\widetilde \Pi$. Since
$\widetilde{\Pi}_1$ and $\widetilde{\Pi}_2$ are mutually orthogonal, $L$
contains at least one root other than its extremal points $\gamma_1$ and
$\gamma_2$. Let $\alpha := \sum_{\gamma \in
L \setminus \{\gamma_1, \gamma_2\} } \gamma$: then $\alpha \in \Phi$, and
$(\alpha,\gamma_i) < 0$, for $i= 1,2$.
Since, for $i=1,2$, $\gamma_i\in \gen_\real V_i$,  $\alpha$ cannot be orthogonal
to all the roots in $V_1$, nor to all the roots in $V_2$.
\end{proof}
\smallskip

This result implies, as a direct corollary, that the intersection of the linear
span of any face with the root system $\Phi$ is an irreducible parabolic
subsystem (Corollary \ref{nonspezzato2}). A second direct consequence
(Corollary \ref{mm}) is that, for any face $F$, 
the vector space $E_F$ is generated by roots, hence that $E_F$ can be
transformed  by some $w$ in $W$, into $\cap_{i\in I}\breve\omega_i^\perp$, for some
$I\subseteq [n]$. This last fact holds for
 the polytope of the orbit of any weight (\cite{V}).

\begin{cor}
\label{nonspezzato2}
Let $F$ be any face of $\mathcal P_{\Phi}$. Then $\gen_\real V_F\cap \Phi$ is an
irreducible parabolic subsystem of $\Phi$. 
\end{cor}

\begin{cor}
\label{mm}
Let $F$ be any face of $\mathcal P_{\Phi}$. Then
the subspace $E_F$ is spanned by roots in $\Phi$.
\end{cor}

\begin{proof}
Let $\Gamma_F$ be the graph having $V_F$  as vertex set
and where $\{\beta, \beta'\}\subseteq  V_F$ is an edge if and only if  $(\beta,
\beta') \neq 0$.  If   $(\beta, \beta') \neq 0$,
$\beta \neq \beta'$, then $(\beta, \beta') > 0$ since the sum of two roots in a
face cannot be a root, and hence $\beta - \beta' $ is a root.
By Theorem \ref{nonspezzato}, $\Gamma_F$ is connected; this implies that $E_F$
is spanned by roots: $\gen_\real(\Phi \cap E_F) = E_F$. 
\end{proof}
\smallskip
 
By Corollary \ref{mm}, together with Lemma \ref{affine}, we obtain a direct proof
that all faces are parabolic, that is, that any face of $\mc P_\Phi$ is the
transformed, by some element of $W$, of a
standard parabolic face $F_I$. As like as Corollary \ref{mm}, this result can be obtained (through our
results in the previous sections) as a special
case of Vinberg's
ones on the polytope of the orbit of a weight. However, we prefer to give 
here a self-contained and well detailed proof.

\begin{rem}\label{w0}
Recall that we denote by $w_0$ the longest element of $W$. It is well known
that 
$w_0(\Pi)=-\Pi$. In particular, for all $\beta\in \Phi$, $\text{\it ht}(\beta)=
\text{\it ht}(-w_0(\beta))$ and hence $w_0(\theta)=-\theta$. It follows that, 
if $w_0(\alpha_i)=-\alpha_{i'}$, then $m_i=m_{i'}$.
Moreover, since $w_0(\Pi\setminus\{\alpha_i\})=-\Pi\setminus\{w_0(\alpha_i)\}$, 
the $W$-invariance of the scalar product implies that, if $w_0(\alpha_i)=-\alpha_{i'}$, then $w_0(\breve\omega_i)=-\breve\omega_{i'}$. 
\end{rem}
\smallskip

\begin{thm}\label{tutteparaboliche2}
All faces of $\mathcal P_\Phi$ are parabolic. 
\end{thm}

\begin{proof}
Let $F$ be a face of dimension $\dim F=n-p$, with $1\leq p\leq n$.  We prove the
claim
by induction on $p$. If $p=1$, then by Corollary \ref{mm} and 
 Proposition \ref{standard3}, we get that $\Phi\cap E_F$ is a parabolic root
susbsystem of $\Phi$ of rank $n-1$. It follows that there exist $w\in W$ and 
 $i\in [n]$ such that $w(E_F)=\breve\omega_i^\perp$. Therefore, there exists
$a\in \real$ such that, for all $\beta\in w(F)$, $(\beta, \breve\omega_i)=a$. 
This forces $a=\pm m_i$. If $a=m_i$ we obtain that $w(F) =F_i$. 
Otherwise, by Remark \ref{w0},  $w_0 w (F)=F_{i'}$, where
$\alpha_{i'}=-w_0(\alpha_i)$.  \par  
Now, we assume $n\geq p>1$ and  let
$\breve\omega_I^\perp= \cap_{i\in I}\breve\omega_i^\perp$, for short.
Let $\tilde F$ be any face such  
that $F\subseteq
\tilde F$ and $\dim \tilde F=n-p+1$. 
By induction, we may assume that 
$\tilde F$ is
standard parabolic, say $\tilde
F=F_I$, with $I\subsetneq [n]$, $\ov I=I$, and hence that 
$E_{\tilde F}=\breve\omega_I^\perp$, by Proposition
\ref{facceparabolichestandard}.
It follows that  $E_{F} \cap \Phi$ is contained in 
$\Phi(\Pi\setminus \Pi_I)$, and is a parabolic subsystem of it, of 
codimension
$1$. Therefore, there exist
$w\in W\la\Pi\setminus \Pi_I\ra$ and $i\in [n]\setminus I$ such that
$w(E_F)=\breve\omega_i^\perp \cap E_{\tilde F}$.
Since, by Proposition \ref{stabilizzatore}, $w(F_I)=F_I$, this implies that $ w
(F)=F_I\cap \{x\mid (x, \breve\omega_i)=a\}$,  for some 
$a\in \real$.\par
Now, let $\eta=\min V_I$ and  $l_i=c_i(\eta)$. Since $\theta\in F_I$, we obtain
that either $a=m_i$ or $a=l_i$. 
In the first case we are done, since then 
$w(F)=w(F_{I\cup\{i\}})$.
Then, we  assume that $a=l_i$. In this case,  
$\eta \in w(F)$, and $w(F) \subseteq \eta+\breve\omega_{I\cup\{i\}}^\perp$.
If $v$ is the longest element in $W\la\Pi\setminus \Pi_I\ra$, then, by
Proposition \ref{standard4} and Remark \ref{radice-minima},  applied to the root
system $\Phi(-(\Pi\setminus \Pi_I)\cup \{\theta\})$, with the root basis
$-(\Pi\setminus \Pi_I)\cup \{\theta\}$, we obtain that $v(\eta)=\theta$, and $v
(\breve\omega_{I\cup\{i\}}^\perp)=\breve\omega_{I\cup\{j\}}^\perp$, for some 
$j$ in $[n]\setminus I$. 
It follows that 
$vw(F)=vw(F_{I\cup\{j\}})$, hence $F$ is
parabolic.
\end{proof}
\smallskip

In particular, the orbits of the facets of $\mathcal P_\Phi$ are in bijection
with the simple roots that do not disconnect the extended Dynkin diagram, when
removed. 
In Table \ref{Tfacets}, we list explicitly, type by type, the simple roots corresponding to the standard parabolic facets. 
In the pictures, the black node corresponds to the affine root, the crossed node and its label denotes the simple root $\alpha_i$  to be removed  and its index $i$.

\begin{table}[]
\caption{Subdiagrams corresponding to the standard parabolic facets}
\label{Tfacets}
\centering
\begin{tabular}{l| l l}
\\
\begin{tikzpicture}
\node{$A_n$};
\end{tikzpicture}
&
\parbox[c]{5cm}{\raisebox{5pt}{
\begin{tikzpicture}[scale=.6]
\An
\draw{(2.8,.2)--(3.2,-.2) (2.8,-.2)--(3.2,.2)}; 
\node[below right] at (2.65,-.1){$\ssc i$\qquad $\ssc (1\leq i\leq n)$};
\end{tikzpicture}
}}
&
\\ \hline\\
\begin{tikzpicture}
\node{$B_n$};
\end{tikzpicture}
&
\parbox[c]{5cm}{\raisebox{10pt}{
\begin{tikzpicture}[scale=.6]
\Bn
\draw{(.05,.8)--(.55,.6) (.05,.6)--(.55,.8)}; 
\node[below] at (.3,.7){$\ssc 1$};
\end{tikzpicture}
}}
&
\parbox[c]{5cm}{\raisebox{10pt}{
\begin{tikzpicture}[scale=.6]
\Bn
\draw{(7.8,.2)--(8.2,-.2) (7.8,-.2)--(8.2,.2)}; 
\node[below] at (8,-.1){$\ssc n$};
\end{tikzpicture}
}}
\\ \hline\\
\begin{tikzpicture}
\node{$C_n$};
\end{tikzpicture}
&
\parbox[c]{5cm}{\raisebox{5pt}{
\begin{tikzpicture}[scale=.6]
\Cn
\draw{(7.8,.2)--(8.2,-.2) (7.8,-.2)--(8.2,.2)}; 
\node[below] at (8,-.1){$\ssc n$};
\end{tikzpicture}
}}
&
\\ \hline\\
\begin{tikzpicture}
\node{$D_n$};
\end{tikzpicture}&
\parbox[c]{5cm}{\raisebox{5pt}{
\begin{tikzpicture}[scale=.6]
\Dn
\draw{(.05,.8)--(.55,.6) (.05,.6)--(.55,.8)}; 
\node[below] at (.3,.7){$\ssc 1$};
\phantom{
\node[below] at (4,-.7){$\ssc r=n-1 \text{ or \ } r=n$};
}
\end{tikzpicture}
}}
&
\parbox[c]{5cm}{\raisebox{5pt}{
\begin{tikzpicture}[scale=.6]
\Dn
\draw{(7.45,.8)--(7.95,.6) (7.45,.6)--(7.95,.8)}; 
\node[below] at (7.75,.65){$\ssc r$};
\node[below] at (4,-.7){$\ssc r=n-1 \text{ or \ } r=n$};
\end{tikzpicture}
}}
\\ \hline\\
\begin{tikzpicture}[scale=.6]
\node at (0,0) {$E_6$};
\end{tikzpicture}&
\parbox[c]{5cm}{
\begin{tikzpicture}[scale=.6]
\Esei
\draw{(-.2,1.2)--(.2,1-.2) (-.2,1-.2)--(.2,1.2)}; 
\node[below] at (0,1-.1){$\ssc 1$};
\end{tikzpicture}
}
&
\parbox[c]{5cm}{
\begin{tikzpicture}[scale=.6]
\Esei
\draw{(4-.2,1.2)--(4+.2,1-.2) (4-.2,1-.2)--(4+.2,1.2)}; 
\node[below] at (4,1-.1){$\ssc 6$};
\end{tikzpicture}
}
\\ 
\\ \hline\\
\begin{tikzpicture}
\node{$E_7$};
\end{tikzpicture}
&
\parbox[c]{5cm}{\raisebox{5pt}{
\begin{tikzpicture}[scale=.6]
\Esette
\draw{(3-.2,-1+.2)--(3+.2,-1-.2) (3-.2,-1-.2)--(3+.2,-1+.2)}; 
\node[below] at (3,-1-.1){$\ssc 2$};
\end{tikzpicture}
}}
&
\parbox[c]{5cm}{\raisebox{5pt}{
\begin{tikzpicture}[scale=.6]
\Esette
\draw{(6-.2,.2)--(6+.2,-.2) (6-.2,-.2)--(6+.2,.2)}; 
\node[below] at (6,-.1){$\ssc 7$};
\node[below] at (3,-1-.1){$\phantom{\ssc 2}$};
\end{tikzpicture}
}}
\\ \hline\\
\begin{tikzpicture}
\node{$E_8$};
\end{tikzpicture}
&
\parbox[c]{5cm}{\raisebox{5pt}{
\begin{tikzpicture}[scale=.6]
\Eotto
\draw{(0-.2,.2)--(0+.2,-.2) (0-.2,-.2)--(0+.2,.2)}; 
\node[below] at (0,-.1){$\ssc 1$};
\node[below] at (2,-1-.1){$\phantom{\ssc 2}$};
\end{tikzpicture}
}}
&
\parbox[c]{5cm}{\raisebox{5pt}{
\begin{tikzpicture}[scale=.6]
\Eotto
\draw{(2-.2,-1+.2)--(2+.2,-1-.2) (2-.2,-1-.2)--(2+.2,-1+.2)}; 
\node[below] at (2,-1-.1){$\ssc 2$};
\end{tikzpicture}
}}
\\ \hline\\
\begin{tikzpicture}
\node{$F_4$};
\end{tikzpicture}
&
\parbox[c]{5cm}{\raisebox{5pt}{
\begin{tikzpicture}[scale=.6]
\Fquattro
\draw{(4-.2,.2)--(4+.2,-.2) (4-.2,-.2)--(4+.2,.2)}; 
\node[below] at (4,-.1){$\ssc 4$};
\end{tikzpicture}
}}
\\ \hline\\
\begin{tikzpicture}
\node{$G_{2}$};
\end{tikzpicture}
&
\parbox[c]{5cm}{
\begin{tikzpicture}[scale=.6]
\Gdue
\draw{(0-.2,.2)--(0+.2,-.2) (0-.2,-.2)--(0+.2,.2)}; 
\node[below] at (0,-.1){$\ssc 1$};
\end{tikzpicture}
}
\\
 \end{tabular}
\end{table}

We can now give a half-space representation
of $\mathcal P_\Phi$ and a more explicit expression for the $f$-polynomial of
$\mathcal P_\Phi$.

\begin{cor}
\label{mezzi-spazi}
The polytope  $\mathcal P_\Phi$ is the intersection of the half-spaces $\{x
\mid (x, w \breve\omega_i )\leq m_i, \forall w\in W, \forall i \in [n] \}$. A
minimal half-space representation is obtained considering only the $i\in [n]$
such that 
$\Phi(\widehat{\Pi}\setminus \{\alpha_i\})$ is irreducible. 
\end{cor}

Since the stabilizer of $\breve\omega_i$ is the parabolic subgroup $W\la \Pi
\setminus \{\alpha_i\} \ra$, in Corollary \ref{mezzi-spazi} we may let $w$ run
over any complete set of left coset representatives of $W$ modulo  $W\la \Pi
\setminus \{\alpha_i\} \ra$. For example, we may consider the set of
the minimal coset representatives $\{w \in W \mid D_r(w)\subseteq \{i\} \}$. 

\begin{cor}\label{fpoly}
For $\Gamma\subseteq\Pi$, let $\widehat \Gamma:=\Gamma\cup\{\alpha_0\}$, and set 
$$\mathcal I:=\{ \Gamma\subseteq \Pi\mid  \widehat\Phi(\widehat \Gamma) 
\text{ is irreducible} \}.$$
The $f$-polynomial of $\mathcal P_\Phi$ is
$$\sum\limits_{\Gamma\in \mathcal I} [W:W\la\Gamma^*\ra]t^{|\Gamma|},$$
where $W\la\Gamma^* \ra$ is the
subgroup of the Weyl group generated by the reflections with respect to 
the roots in $\Gamma^* := \Gamma \cup (\widehat \Gamma^{\perp} \cap \Pi)$. 
\end{cor}

\begin{proof}
 It follows by Proposition \ref{chiusi} , Proposition
\ref{stabilizzatore}, and Theorem \ref{tutteparaboliche2}.
\end{proof}
\smallskip

We end this section with the following elegant statement that can
be obtained as an immediate
consequence of Theorems
\ref{tutteparaboliche2} and Proposition \ref{biezione-st-di} and 
is the
$\lambda=\theta$ case of Proposition 3.2
of \cite{V}.

\begin{thm}\label{combinatorio}
The orbits of the faces of $\mathcal P_\Phi$ under the action of $W$ are in
bijection 
with the connected subdiagrams of the extended Dynkin diagram that contain the
affine node. Equivalently, the orbits of the faces are in bijection with the
standard parabolic irreducible root subsystems of $\widehat{\Phi}$ that are not
included in~$\Phi$.
\end{thm}

\section{Faces of low dimension}

We know that the short roots are convex linear combinations of the long ones
(Remark \ref{cortesuperflue}). 
Here we first see that if a short root lies on some face $F$, then it is a
convex combination of two long roots in $F$. Moreover, we prove that if some
short root lies on some face, then the faces of lowest dimension containing the
short roots form a single orbit of $W$, and we determine the standard parabolic
face in it. Then, we give a very peculiar property of the $1$-skeleton
of $\mathcal P_\Phi$, whose edges are made up of roots.

\begin{lem}\label{corteinterne}
Let $F$ be a face of $\mathcal P_\Phi$.
Then $F$ contains at least one long root. Moreover, if  $F$ contains some short
root, then the ratio between the squared lengths of the long and the short roots
is $2$, and each short root in $F$ is the halfsum of two long roots in $F$. \par
\end{lem}

\begin{proof}
We may assume $F$ standard parabolic, say $F=F_I$, with $I\subseteq [n]$.
The first assertion is clear, since $\theta$ belongs to any standard parabolic face.
If $F$ contains also some short root,  then, by Theorem \ref{nonspezzato}, there
exists two non-orthogonal roots of different lengths, say $\beta$ and $\beta'$,
in $F$, with $\frac{(\beta, \beta)}{(\beta', \beta')}=r>1$. Then
$s_{\beta'}(\beta)=\beta-r\beta'$ is a root  and, since $\beta, \beta'\in V_I$,
$c_i( s_{\beta'}(\beta))= m_i-rm_i$ for all $i\in I$. This implies that $r=2$,
$c_i( s_{\beta'}(\beta))=-m_i$, and $-s_{\beta'}(\beta)=2\beta'-\beta\in F$.
Therefore $\beta'=\frac{1}{2}(\beta-s_{\beta'}(\beta))$, a convex combination of
the long roots $\beta$ and $-s_{\beta'}(\beta)$, both lying in $F_I$.
\end{proof}
\smallskip
The following proposition tells
how far the short roots are 
from being vertices. Recall that $\theta_s$ denotes the highest short root of
$\Phi$ and $I(\theta_s)=\{i \in [n] \mid c_i(\theta_s)= m_i\}$.

\begin{pro}\label{tetacorta}
Let $\Phi$ be not simply laced. 
Then the minimal dimension of the faces of  $\mathcal P_\Phi$ containing
a short root is $n - |I(\theta_s)|$. Moreover, the faces of dimension $n-|I(\theta_s)|$
containing a short root form a single $W$-orbit.
\end{pro}

\begin{proof}
 Any standard parabolic face $F$ containing a short root contains  $\theta_s$,
since $V_F$ is a dual order ideal in the root poset and $\theta_s$ is greater
than any other short root.  
By Theorem \ref{tutteparaboliche2}, it is enough to find the dimension of the
minimal standard parabolic face of  
$\mathcal P_\Phi$ containing $\theta_s$.  This is clearly $F_{I(\theta_s)}$, and its
dimension is $n - |I(\theta_s)|$ by Propositions \ref{ieoviuguale} and \ref{facceparabolichestandard}. The last
statement follows since the action of $W$ is transitive on the set of short
roots and the intersection of two distinct faces of dimension $n-|I(\theta_s)|$ does not
contain any short root.
\end{proof}
\smallskip

Suppose that $\Phi$ be not simply  laced. 
By Lemma \ref{corteinterne}, we already know that, if the ratio between the
squared lengths of the long and the short roots is 3,
 then the short roots lie in the inner  part of $\mathcal P_\Phi$, and hence $I(\theta_s)$
must be empty. On the other hand, if the ratio between 
the squared lengths of the long and the short  roots is 2, by a case-by-case
argument we see that $I \neq \emptyset$, and hence  
the short roots are always on the border of $\mathcal P_\Phi$.
In particular, for types $B_n$ and $F_4$, the minimal faces
containing short roots are the facets, while, for type $C_n$, the short roots
lie on the edges.
\par
We end this section with a direct description of the 1-skeleton of 
$\mathcal P_{\Phi}$, whose edges are, in fact, made up of roots.

\begin{cor}
\label{mm3}
Let $F$ be a face of $\mathcal P_{\Phi}$
of dimension 1. Then the roots in $F$ form a string with 2 or 3 roots: 
if $\gamma$ and $\gamma'$ are the vertices of $F$, then either 
\begin{enumerate}
\item $\gamma$ and $\gamma'$ are not perpendicular, there are no other roots in
$F$, and $\gamma - \gamma'$ is a root, or 
\item  $\gamma$ and $\gamma'$ are perpendicular, there is only a third root
$\gamma''$ in $F$ ($\gamma''$ short), and  
$\gamma - \gamma'' = \gamma'' - \gamma'$ is a root (so that $\gamma - \gamma'$
is twice a root). 
\end{enumerate}
\end{cor}

\begin{proof}
Let $\gamma_0, \gamma_1, \ldots, \gamma_r$ be the roots in $F$.  Since $F$ has
dimension 1, we may assume that $\gamma_i = \gamma_0 + k_i \beta $, 
$k_i \in \mathbb R$, for all $i \in [r]$, with $1=k_1 < k_2 < \cdots < k_r$ and
$\real \beta= E_F$. 
Recall that two
roots in a face cannot have negative scalar product. 
Using Theorem \ref{nonspezzato}, we easily show that 
$(\gamma_i, \gamma_{i+1})>0$, 
for all $i\in [0,r-1]$. Then $\beta \in \Phi$ and  $\gamma_i - \gamma_{i-1}$ is
both a root and a positive multiple of $\beta$, for all $i \in [r]$, by
Proposition \ref{standard}. 
Then $\{\gamma_0, \gamma_1, \ldots, \gamma_r\}$ is the $\beta$-string through
$\gamma_0$.
It is well known that the strings can have only 1,2,3, or 4 roots.  In this
case, it clearly does not have 1 root and it cannot have 
4 roots since in a string of 4 roots the product of the first with the forth one
is negative. Hence we have done, cause the string with 
3 roots are necessarily  of the type in the statement (see \cite{Bou}, Ch. VI, \S 1, n. 4).  
\end{proof}
\smallskip

Any face $F$ of $\mathcal P_{\Phi}$ is a facet of the root polytope $\mathcal
P_{\Phi \cap \gen_\real F} \subseteq \gen_\real F$.  Since the root
system $\Phi \cap \gen_\real F$ is irreducible by Theorem
\ref{nonspezzato}, all
possible types of faces of dimension $k$ already occur in $(k+1)$-dimensional
polytopes of the type studied in this work. In particular, Corollary \ref{mm3}
can also be proved by an immediate case-by-case verification of the
2-dimensional root polytopes.
\par
Note that the standard parabolic edges of $\mathcal P_{\Phi}$ are made up of those roots
$\alpha_k$ which are adjacent to $\alpha_0$ in the affine Dynkin diagram of
$\Phi$ (Proposition \ref{chiusi} and Corollary \ref{mm3}).
Then, by Theorem
\ref{tutteparaboliche2}, the 1-skeleton of $\mathcal
P_{\Phi}$ is made up of the long or the short roots, depending on the
length of the simple roots adjacent to 
$\alpha_0$. In particular, in type $C$, the short roots occur, and in all the other cases the long ones.

\section{Acknowledgements}

The authors would like to thank the referee for his/her useful comments and suggestions. In particular, Proposition \ref{carat-min} has been written to address one of his/her questions.

\end{document}